\documentclass[a4paper]{amsart}

\usepackage{mathrsfs,amssymb}
\usepackage{multicol}\multicolsep=0pt
\usepackage{pstricks,pst-node}
\usepackage[backref=page]{hyperref}

\usepackage[sort]{cite}

\def\crulefill{\leavevmode\leaders\hrule height 1pt\hfill\kern 0pt}
\long\def\QUERY#1{%
\leavevmode\newline%
\noindent$\star\star\star$\thinspace\textsf{Comment/Query}\crulefill\newline%
   \space #1\newline\hbox to 120mm{\crulefill}$\star\star\star$\newline
}

 \numberwithin{equation}{section}
\theoremstyle{definition}
\newtheorem{number1}[equation]{}
\newtheorem{Defn}[equation]{Definition}

\newtheorem{Remark}[equation]{Remark}
\theoremstyle{plain}
\newtheorem{Prop}[equation]{Proposition}
\newtheorem{Theorem}[equation]{Theorem}
\newtheorem{Assumption}[equation]{Assumption}
\newtheorem{Lemma}[equation]{Lemma}
\newtheorem{Cor}[equation]{Corollary}

\makeatletter
\def\enumerate{\begingroup\ifnum\@enumdepth>3\@toodeep\else
      \advance\@enumdepth\@ne
      \edef\@enumctr{enum\romannumeral\the\@enumdepth}%
      \topsep\z@\parskip\z@
      \list{\csname label\@enumctr\endcsname}
        {\@nmbrlisttrue\let\@listctr\@enumctr
         \parsep\z@\itemsep\z@\topsep\z@
         \setcounter{\@enumctr}{0}
         \def\makelabel##1{\hss\llap{\rm ##1}}
       }\fi}

\makeatother

\let\epsilon=\varepsilon
\def\({\big(}
\def\){\big)}

\def\bfs{\mathbf s}
\def\bft{\mathbf t}
\def\bfu{\mathbf u}
\def\bfv{\mathbf v}

\def\t{\mathfrak t}
\def\u{\mathfrak u}
\def\v{\mathfrak v}

\def\ss{\mathbf s}

\def\m{\mathfrak m}
\def\floor#1{\lfloor\tfrac#1\rfloor}
\def\UPD{\mathscr{T}^{ud}}
\def\Std{\mathscr{T}^{s}}
\def\SStd{\mathscr{T}^{ss}}

\def\bfS{\mathbf S}

\def\N{\mathbb N}

\def\Z{\mathbb Z}

\def\0{\underline{0}}
\def\bu{\mathbf u}
\def\B{\mathscr B}

\def\Bcal{\mathfrak B}
\def\Dcal{\mathcal D}
\def\Ef{{\mathcal E}_f}
\def\G{\mathcal G}
\def\H{\mathscr H}
\let\proj=\varepsilon
\def\Nrf{\mathbb N_r^{f, n}}
\def\Nrl{\mathbb N_r^{\ell, n-1}}

\def\Sym{\mathfrak S}
\def\W{\mathscr B}

\def\Wlam{\W_{r,n}^{\gdom(f, \lambda)}}

\def\simk{\overset{k}\sim}
\def\simn{\overset{n-1}\sim}

\DeclareMathOperator{\Rad}{Rad}

\def\F{\mathcal F}
\def\G{\mathcal G}
\let\gdom\rhd
\let\gedom\unrhd

\def\m{\mathfrak m}
\def\floor#1{\lfloor\tfrac#1\rfloor}
\def\UPD{\mathscr{T}^{ud}}
\def\Std{\mathscr{T}^{std}}

\def\Wmuieq{S^{\gedom\mu_\lambda(i)}}

\def\ss{\mathfrak s}

\def\tt{\mathfrak t}
\def\uu{\mathfrak u}
\def\vv{\mathfrak v}

\def\s{\mathbf s}

\def\t{\mathbf t}
\def\u{\mathbf u}
\def\v{\mathbf v}

\def\bfs{\mathbf s}
\def\bft{\mathbf t}
{\catcode`\|=\active
  \gdef\set#1{\mathinner{\lbrace\,{\mathcode`\|"8000%
                                   \let|\midvert #1}\,\rbrace}}
  \gdef\seT#1{\mathinner{\Big\lbrace\,{\mathcode`\|"8000%
                                   \let|\midverT #1}\,\Big\rbrace}}
}
\def\midvert{\egroup\mid\bgroup}
\def\midverT{\egroup\,\Big|\,\bgroup}

\def\Set[#1]#2|#3|{\Big\{\ #2\ \Big| \
           \vcenter{\hsize #1mm\centering #3}\Big\}}

\def\map#1#2{\,{:}\,#1\!\longrightarrow\!#2}

\title{The representations of cyclotomic BMW algebras, II}

\author{Hebing Rui}
\address{H. Rui, Department of Mathematics, East China Normal University, %
         200241 Shanghai, P.R. China.}
\email{hbrui@math.ecnu.edu.cn}

\author{Mei Si  }

\address{M. Si, Department of Mathematics, Shanghai Jiaotong University, %
         200240, Shanghai, P.R. China.}
\email{simeism@hotmail.com}
\thanks{The first author is supported in part by
NSFC and NCET-05-0423.}
\date{July 25, 2008}
\begin{document}
\sloppy \maketitle


\begin{abstract} In this paper, we  go on Rui-Xu's  work on
 cyclotomic Birman-Wenzl algebras
$\W_{r, n}$ in \cite{RX}. In particular, we use the representation
theory of cellular algebras in \cite{GL} to classify the irreducible
$\W_{r, n}$-modules for all positive integers $r$ and $n$.  By
constructing  cell filtrations for all cell modules of $\W_{r, n}$,
we compute the discriminants associated to all  cell modules for
$\W_{r, n} $.  Via such discriminats together with induction and
restriction functors given in section~5, we determine explicitly
when  $\W_{r, n}$ is  semisimple over a field.  This generalizes our
previous result on  Birman-Wenzl algebras in \cite{RS1}.
\end{abstract}

\section{Introduction}

Let $\W_{r, n}$ be the  cyclotomic Birman-Wenzl algebras defined
in~\cite{HO:cycBMW}. Motivated by Ariki, Mathas and Rui's work on
cyclotomic Nazarov--Wenzl algebras~\cite{AMR}, Rui and Xu~\cite{RX}
proved that $\W_{r, n}$ is cellular over $R$ for all positive odd
integers $r$ under the so-called $\u$-admissible conditions (see the
assumption~\ref{admiss}). Moreover, they  have classified the
irreducible $\W_{r, n}$-modules.

In this paper, we will prove that $\W_{r, n}$ is cellular over $R$
for all positive integers $r$ under the $\u$-admissible conditions.
By using arguments in \cite{RX}, we classify the irreducible $\W_{r,
n}$-modules over an arbitrary field. This completes the
classification of irreducible $\W_{r, n}$-modules over a field. We
remark that Yu~\cite{Yu1} first  proved that $\W_{r, n}$ is cellular
over $R$ under the similar conditions. However, she did assume that
the parameter $\omega_0$, which is  given in  Definition~\ref{waff},
is invertible when she proved that $\W_{r, n}$ is cellular.

Given a cell module $M$ of $\W_{r, n}$.  Following \cite{RS1}, we
construct a $\W_{r, n-1}$-filtration for $M$. Via it, we  construct
an $R$-basis for $M$, called JM-basis in the sense of \cite{M:semi}.
This enables us to use standard arguments in \cite{M:semi} to
construct an orthogonal basis for $M$ under so called
\textsf{separate condition} in the sense of \cite{M:semi}. The key
is that the Gram determinants associated to $M$ which are defined by
the JM-basis and the previous orthogonal basis are the same. We will
give a recursive formula to compute the later determinant.

Motivated by \cite{DWH},  we  construct restriction functor $\F$ and
induction functor $\G$ which set up a relationship between the
category of $\W_{r, n}$-modules and the category of $\W_{r,
n-2}$-modules. Via $\F$ and $\G$ together with certain explicit
formulae on Gram determinants, we determine explicitly when $\W_{r,
n}$ is semisimple over a field.

We organize the paper as follows. In Section~2, we prove that
$\W_{r, n}$ is cellular over $R$ for all positive integers $r$ and
$n$. We also classify the irreducible $\W_{r, n}$--modules. In
section~3, we construct the JM-basis and an orthogonal basis for
each cell module of $\W_{r, n}$. In section~4, we compute the
discriminants associated to   all cell modules of $\W_{r, n}$.
Restriction functor $\F$ and induction functor  $\G$ will be
constructed in section~5. In section~6, we  determine explicitly
when $\W_{r, n}$ is semisimple over an arbitrary field.

\section{The  cyclotomic Birman-Wenzl algebras}

Throughout the paper, we fix  two positive integers $r$ and $n$. Let
$R$ be a commutative ring which contains the identity $1$ and
invertible elements $q^{\pm 1}$, $u_1^{\pm 1}, \dots, u_r^{\pm 1}$,
$\varrho^{\pm 1}, \delta^{\pm 1}$ such that $\delta=q-q^{-1}$ and
$\omega_0=1-\delta^{-1}(\varrho-\varrho^{-1})$.

\begin{Defn}\cite{HO:cycBMW}\label{waff} The  cyclotomic Birman-Wenzl algebra $\W_{r, n}$
is the  unital associative $R$--algebra generated  by $\set{T_i,E_i,
X_j^{\pm 1}|1\le i<n \text{ and }1\le j\le n } $ subject to the
following relations:
\begin{enumerate}
\item $X_i X_{i}^{-1}=X_{i}^{-1}X_i=1$ for $1\le i\le n$.
    \item (Kauffman skein relation )
$1=T_i^2-\delta T_i +\delta \varrho E_i$, for $1\le i<n$.
    \item (braid relations)
\begin{enumerate}
\item $T_iT_j=T_jT_i$ if $|i-j|>1$,
\item $T_iT_{i+1}T_i=T_{i+1}T_iT_{i+1}$, for $1\le i<n-1$,
\item $T_iX_j=X_jT_i$ if $j\ne i,i+1$.
\end{enumerate}
    \item (Idempotent relations)
$E_i^2=\omega_0E_i$, for $1\le i<n$.
    \item (Commutation relations)  $X_iX_j=X_jX_i$, for $1\le i,j\le n$.

\item (Skein relations)
\begin{enumerate}\item      $T_iX_i-X_{i+1}T_i=\delta X_{i+1}
(E_i-1)$,   for $1\le i<n$,
    \item     $X_iT_i-T_iX_{i+1}=\delta (E_i-1) X_{i+1}$, for $1\le i<n$.
    \end{enumerate}
    \item (Unwrapping relations)
        $E_1X_1^aE_1=\omega_a E_1$, for $a\in \mathbb Z$.
    \item (Tangle relations)
\begin{enumerate}
\item $E_iT_i=\varrho E_i=T_iE_i$, for $1\le i\le n-1$,

\item $E_{i+1}E_i=E_{i+1}T_iT_{i+1}=T_iT_{i+1} E_i$, for $1\le i\le n-2$.
\end{enumerate}

    \item (Untwisting relations)\begin{enumerate} \item
        $E_{i+1}E_iE_{i+1}=E_{i+1}$  for $1\le i\le n-2$,
        \item $E_iE_{i+1}E_i=E_i$, for $1\le i\le
        n-2$.\end{enumerate}
    \item (Anti--symmetry relations) $E_iX_iX_{i+1}=E_i=X_iX_{i+1}E_i$, for $1\le i<n$.
\item (Cyclotomic relation) $(X_1-u_1)(X_1-u_2)\cdots(X_1-u_r)=0$
\end{enumerate}
\end{Defn}

For each $x\in R$, let   $$\gamma_r(x)=\begin{cases} 1, & \text{if
$2\nmid r$,}\\
-x, &\text{if $2\mid r$}.\\
\end{cases}
$$

In the remainder of this paper, We  use $\u$ (resp. $\Omega$) to
denote $(u_1, u_2, \dots, u_r)$ (resp. $\{\omega_a\mid a\in \mathbb
Z\}$). In order to show that  $\W_{r, n}$ is free over $R$, Rui and
Xu  introduced the \textsf{$\u$-admissible conditions} in
\cite[3.15]{RX} as follows.

\begin{Assumption}\label{admiss}  $\Omega\cup\{\varrho\}$ is called
\textsf{$\u$-admissible} if $$\varrho^{- 1}= \alpha \prod_{\ell=1}^r
u_{\ell}, \text{ and } \omega_{a} = \sum\limits_{j=1}^r
u_{j}^{a}\gamma_{j}, \forall a\in \mathbb Z$$  where
\begin{itemize} \item [(1)] $\gamma_{i} = (\gamma_r(u_i) + \delta^{-1}\varrho(u_{i}^{2}
- 1) \prod \limits_{j\neq i}u_{j})\prod\limits_{j\neq
i}\frac{u_{i}u_{j} - 1}{u_{i} - u_{j}}$,
\item [(2)] $\alpha\in \{1, -1\}$ if $2\nmid r$ and $\alpha\in \{q^{-1}, -q\}$, otherwise.
\item[(3)]
 $\omega_{0}=
\delta^{-1}\varrho(\prod\limits_{\ell=1}^r u_{\ell}^2 - 1) +
1-\frac{(-1)^r +1}{2}\alpha^{-1}\varrho^{-1} $.
\end{itemize}\end{Assumption}

Note that there are infinite equalities in the definition of
$\u$-admissible conditions in Assumption~\ref{admiss}. It has been
proved  in \cite[3.17]{RX} that $\omega_j, \forall j\in\mathbb Z,$
are determined by $\omega_i$, $0\le i\le r-1$. Furthermore, all
$\omega_i$  are elements in $\mathbb Z[u_1^{\pm 1}, \dots, u_r^{\pm
1}, q^{\pm 1}, \delta^{-1}]$ \cite[3.11]{RX}. Therefore,
  $\omega_i\in R$ for all $ i\in \mathbb
Z$ if they are given in the Assumption~\ref{admiss}.

\textsf{In the remainder  of this  paper, unless otherwise stated,
we always keep the Assumption~\ref{admiss} when we discuss $\W_{r,
n}$ over $R$. }

It has been proved in  \cite{RX} that  $\W_{r, n}$ is a free
$R$-module with rank $r^n (2f-1)!!$ when $r$ is odd. We will prove
that $\W_{r, n}$ is  cellular over $R$ with rank $r^n (2f-1)!!$ when
$r$ is  even.  We start by recalling the definition of Ariki-Koike
algebras in \cite{AK}.

The  Ariki-Koike algebra ~\cite{AK} $\H_{r,n}(\bu):=\H_{r, n}$  is
the unital associative $R$-algebra generated by $y_1, \dots, y_n $
and $g_1, g_2, \dots, g_{n-1}$ subject to the following relations:
\begin{enumerate}
\item $(g_i-q)(g_i+q^{-1})=0$, if $1\le i\le n-1$,
\item $g_ig_j=g_jg_i$, if $|i-j|>1$,
\item $g_ig_{i+1}g_i=g_{i+1}g_ig_{i+1}$, for $1\leq i<n-1$,
\item $g_iy_j=y_jg_i$, if $j\neq i,i+1$,
\item $y_iy_j=y_jy_i$, for $1\leq i,j\leq n$,
\item $y_{i+1}=g_iy_ig_i$,
for $1\leq i\leq n-1$,
\item $(y_1-u_1)(y_1-u_2)\dots(y_1-u_r)=0$.
\end{enumerate}

\medskip
\def\Ef{{\mathcal E}}
Let $\Ef_n=\W_{r,n}E_1\W_{r,n}$ be the two-sided ideal of $\W_{r,
n}$ generated by~$E_1$. It is proved in \cite[5.2]{RX} that
  $\H_{r,n}\cong\W_{r,n}/\Ef_n$. The corresponding $R$-algebraic
isomorphism is determined by
$$\proj_{n}: g_i \longmapsto T_i+\Ef_n, \text{ and }\quad y_j\longmapsto X_j+\Ef_n,$$
for $1\le i<n$ and $1\le j\le n$.
\medskip

Let $\mathfrak S_n$ be the symmetric group on $\{1, 2, \dots, n\}$.
Then $\mathfrak S_n$ is generated by $s_i=(i, i+1)$, $1\le i\le
n-1$. If $w=s_{i_1}\cdots s_{i_k}\in \mathfrak S_n$ is a reduced
expression of $w$, then we write   $T_w=T_{i_1}T_{i_2}\cdots
T_{i_k}\in \W_{r, n}$.  It has been pointed out in \cite{RX} that
$T_w$ is independent of a reduced expression of $w$. We denote by
\begin{equation}\label{nr} \N_r=\left\{i\in \mathbb Z\mid
-\lfloor\frac r 2\rfloor+\frac 12 (1+(-1)^r) \le i\le \lfloor \frac
r 2\rfloor\right\}.\end{equation}

Given a non-negative integer $f$ with  $f\le \lfloor n/2\rfloor$.
Following \cite[5.5]{RX}, we define  \begin{equation} \label{co1}
\Dcal_{f, n}=\Set[40]s_{n-2f+1,i_f}s_{n-2f+2,j_f}\cdots
s_{n-1,i_1}s_{n,j_1}|
      $1\leq i_f<\cdots<i_1\leq n, \atop 1\leq i_k<j_k\leq n-2k+2, 1\leq k\leq f$
      |, \end{equation}
      where $$s_{i,j}=\begin{cases}  s_{i-1}s_{i-2}\cdots s_j, & \text{if
$i>j$,}\\
s_is_{i+1}\cdots s_{j-1}, &\text{if $i<j$,}\\
1, & \text{if $i=j$.}\\
\end{cases}
$$
 Let $\Bcal_f\subset \mathfrak S_n$ be the subgroup  generated
by
 $s_{n-2i+2}
         s_{n-2i+1}s_{n-2i+3} s_{n-2i+2}$, $2\le i\le f$, and
         $s_{n-1}$.
Then $\Dcal_{f, n}$ is a right coset representatives for $\mathfrak
S_{n-2f}\times\Bcal_f$ in $\mathfrak S_{n}$ (see e.g. \cite{RX}).

 For each
$d=s_{n-2f+1,i_f}s_{n-2f+2,j_f}\cdots
      s_{n-1,i_1}s_{n,j_1}\in \Dcal_{f, n}$,
let $\kappa_d$ be the
      $n$-tuple $(k_1,\dots,k_n)$ such that $k_i\in\N_r$ and
$k_i\ne0$ only for $i=i_1,i_2,\dots,i_f$. Note that $\kappa_d$ may
be equal to $\kappa_e$ although $e\neq d$ for $e, d\in \Dcal_{f,
n}$. We set $X^{\kappa_d}=\prod_{i=1}^n X_i^{\kappa_i}$. By
Definition~\ref{waff},
\begin{equation}\label{td}
T_dX^{\kappa_d}=T_{n-2f+1,i_f}X_{i_f}^{\kappa_{i_
f}}T_{n-2f+2,j_f}\cdots T_{n-1,i_1}X_{i_1}^{\kappa_{i_1}}T_{n,j_1},
\end{equation} where $T_{i, j}=T_{s_{i, j}}$.
For convenience, let
  \begin{equation}\label{nrf} \Nrf =\{\kappa_d \mid d\in \Dcal_{f,
  n}\}.\end{equation}

Recall that  a \textsf{composition}  $\lambda$ of $m$ is a sequence
of non--negative integers $(\lambda_1,\lambda_2,\dots)$ such that
$|\lambda|:=\lambda_1+\lambda_2+\cdots=m$. $\lambda$ is called a
\textsf{partition} if $\lambda_i\ge \lambda_{i+1}$ for all positive
integers $i$. Similarly, an \textsf{$r$-partition} (resp.
$r$-composition)  of $m$ is an ordered $r$--tuple
$\lambda=(\lambda^{(1)},\dots,\lambda^{(r)})$ of partitions (resp.
compositions)  $\lambda^{(s)}$, $1\le s\le r$,  such that
$|\lambda|:=|\lambda^{(1)}|+\dots+|\lambda^{(r)}|=m$. In the
remainder of this paper, we use multipartitions (resp.
multicompositions) instead of $r$--partitions (resp.
$r$-compositions). Let $\Lambda_r^+(m)$ (resp. $\Lambda_r(m)$ ) be
the set of all multipartitions (resp. multicompositions) of $m$.

It is known that both $\Lambda_r^{+}( m)$ and $\Lambda_r(m)$ are
posets with the dominance order $\unrhd$ defined on them.  We have
$\lambda\trianglelefteq\mu$ if  $$\sum_{j=1}^{i-1} |\lambda^{(j)}|
+\sum_{k=1}^l \lambda_k^{(i)} \le \sum_{j=1}^{i-1}
|\mu^{(j)}|+\sum_{k=1}^l \mu_k^{(i)} $$ for  $1\le i\le r$ and $l\ge
0$. We write $\lambda\vartriangleleft\mu$ if
$\lambda\trianglelefteq\mu$ and $\lambda\neq \mu$. Let $$\Lambda_{r,
n}^+=\{(k, \lambda)\mid 0\le k\le \lfloor n/2\rfloor, \lambda\in
\Lambda_r^+(n-2k)\}.$$  Then  $\Lambda_{r, n}^+$ is a poset with
$\unrhd$ as the  partial order on it.
 More explicitly, $(k, \lambda)\unrhd(\ell, \mu)$
for $(k, \lambda), (\ell, \mu)\in \Lambda_{r, n}^+$ if either
$k>\ell$ in the usual sense or $k=\ell$ and $\lambda\unrhd\mu$. Here
$\unrhd$ is the dominance order defined on  $\Lambda_r^+(n-2k)$.

The Young diagram $Y(\lambda)$ of a partition $\lambda=(\lambda_1,
\lambda_2, \cdots)$ is a collection of boxes arranged in
left-justified rows with $\lambda_i$ boxes in the $i$-th row of
$Y(\lambda)$. A $\lambda$-tableau $\t$  is obtained from
$Y(\lambda)$ by inserting $\{1, \dots, n\}$ into each box of
$Y(\lambda)$ without repetition. If the entries in $\t$ increase
from left to right in each row and from top to bottom in each
column, then $\t$ is called a standard $\lambda$-tableau.

If $\lambda=(\lambda^{(1)}, \dots, \lambda^{(r)})\in
\Lambda_r^+(n)$, then the  Young diagram $Y(\lambda)$ is an ordered
Young diagrams $(Y(\lambda^{(1)}),  \dots, Y(\lambda^{(r)}))$. In
this case, a $\lambda$-tableau $\bft$ is $(\t_1, \dots, \t_r)$ where
each $\t_i, 1\le i\le r$ is  a  $\lambda^{(i)}$-tableau. If the
entries in each $\t_i$ increase from left to right in each row and
from top to bottom in each column, then $\t$ is called standard.
Let $\Std(\lambda)$ be the set of all standard $\lambda
$-tableaux.

Suppose $\lambda\in \Lambda_r^+(n)$.  It is well-known that
$\Std(\lambda)$ is a poset with dominance order $\unrhd$  on it. For
each $\bfs\in \Std(\lambda)$ and a positive integer $i\le n$, let
$\bfs\downarrow_i$ be obtained from $\bfs$ by deleting  all entries
in $\bfs$ greater  than $i$. Let $\bfs_i$ be the multipartition of
$i$ such that $\bfs\!\!\downarrow_i$ is the $\bfs_i$-tableau. Then
$\bfs\trianglerighteq \bft$ if and only if
$\bfs_i\trianglerighteq\bft_i$ for all $i, 1\le i\le n$. Write
$\bfs\rhd \bft$ if $\bfs\trianglerighteq\bft$ and $\bfs\neq \bft$.

It is well-known that $\mathfrak S_n$ acts on a $\lambda$-tableau by
permuting its entries. Let $\bft^\lambda$ be the $\lambda$-tableau
obtained from $Y(\lambda)$ by adding $1, 2, \cdots, n$ from left to
right along the  rows of $Y(\lambda^{(1)})$, $Y(\lambda^{(2)})$,
etc. For example, if $\lambda=((3,2),(2,1),(1,1))\in
\Lambda_3^+(10)$, then
$$\begin{picture}(0,0)(20,2.5)\put(-40,0){$\bft^\lambda=$}
\put(-15,0){(}
\put(-5,0){\framebox(11,11)[cc]{}\framebox(11,11){}\framebox(11,11)[cc]{}}\put(-5,-11){\framebox(11,11)[cc]{}\framebox(11,11)[cc]{}}
\put(48,0){\framebox(11,11)[cc]{}\framebox(11,11)[cc]{}}\put(48,-11){\framebox(11,11)[cc]{}}
\put(90,0){\framebox(11,11)[cc]{}}\put(90,-11){\framebox(11,11){}}\put(115,0){)}
\put(-1,1){1}\put(9,1){2}\put(20,1){3}\put(-1,-10){4}\put(9,-10){5}\put(52,1){6}\put(63,1){7}\put(52,-10){8}\put(94,1){9}\put(91,-10){10}
\put(35,-6){,}\put(77,-6){,}
\end{picture}$$
\medskip

Let $\mathfrak S_\lambda$ be the Young subgroup associated to the
multipartition $\lambda$. Then $\mathfrak S_\lambda$ is the row
stabilizer of $\bft^\lambda$.
 Let $a_i=\sum_{j=1}^i
        |\lambda^{(j)}|$, $1\le i\le r$ and $a_0=0$.
For each $\lambda$-tableau $\t$, there is a unique element, say
$d(\t)\in \mathfrak S_n$,  such that  $\t=\t^\lambda d(\t) $.
 Suppose that   $\s,\t\in \Std(\lambda)$ where $\lambda\in\Lambda_r^+(n-2f)$ for some non-negative
integer $f\le \floor{n2}$. It is defined in \cite[5.7]{RX} that
\begin{equation} \label{JMe} M_{\s\t}=T_{d(\s)}^\ast \cdot \prod_{s=2}^r
      \prod_{i=1}^{a_{s-1} }(X_i-u_s)
        \sum_{w\in\Sym_\lambda} q^{l(w)}T_w \cdot
        T_{d(\t)},\end{equation} where $\ast$  is the $R$-linear anti-involution on $\W_{r, n}$, which
fixes $T_i$ and $X_j$, $1\le i\le n-1$ and $1\le j\le n$. Note that
\begin{equation}\label{jme}\m_{\s\t}:=\epsilon^{-1}_{n-2f} (
M_{\s\t}+\Ef_n)\end{equation} is the Murphy basis element for
Ariki-Koike algebra $\H_{r, n-2f}$ in \cite{DJM:cyc}.

We define $M_\lambda=M_{\t^\lambda\t^\lambda}$ and  $E^{f,
n}=E_{n-1}E_{n-3}\cdots E_{n-2f+1}$ and $\W_{r,n}^f=\W_{r,n}E^{f,
n}\W_{r,n}$ for each non-negative integer $f\le \lfloor n/2\rfloor$.
Therefore,  there is  a filtration of two-sided ideals of $\W_{r,n}$
as follows:
\begin{equation}\label{filt}\W_{r,n}=\W_{r,n}^0\supset \W_{r,n}^1\supset\dotsi
     \supset \W_{r,n}^{\floor{n2}}\supset
     0.\end{equation}

\begin{Defn}\label{cellmodule} Suppose that $0\le
f\le\floor{n2}$ and   $\lambda\in \Lambda_r^+(n-2f)$. Define
$\W_{r,n}^{\unrhd(f, \lambda)}$ to be the two--sided ideal of
$\W_{r,n}$ generated by   $\W_{r,n}^{f+1}$ and $S$ where
$$S=\set{E^{f, n}M_{\s\t}|\s,\t\in \Std(\mu) \text{ and }
    \mu\in\Lambda_r^+(n-2f)\text{ with
    }\mu\trianglerighteq\lambda} .$$  We
also define $\W_{r, n}^{\rhd (f, \lambda)} =\sum_{\mu\gdom
\lambda}\W_{r,n}^{\gedom(f, \mu)}$, where in the sum
$\mu\in\Lambda_r^+(n-2f)$.\end{Defn}

By Definition~\ref{waff}, there is a natural homomorphism from
$\W_{r, m}$ to $\W_{r, n}$ for positive integers $m\le n$. Let
$\W_{r, m}'$ be the image of $\W_{r, m}$ in $\W_{r, n}$.
 The following result, which plays the key role,
has been proved by Yu without assuming that $\omega_0$ is
invertible~\cite{Yu1}.

\begin{Lemma}\label{Yu2.7} $N$ is a right $\W_{r,
n}$-module if   $N$ is the $R$-submodule generated by
$\W_{r,n-2f}'E^{f, n}T_dX^{\kappa_d}$, for all  $d\in\Dcal_{f,n}$
and $\kappa_d\in\Nrf$.\end{Lemma}

\begin{Prop}(cf. \cite[5.10]{RX}) \label{key}Suppose that $\s\in \Std(\lambda)$. We define
$\Delta_\s(f,\lambda)$ to be the $R$-submodule of
$\W_{r,n}^{\unrhd(f, \lambda)}/\W_{r,n}^{\rhd(f, \lambda)}$  spanned
by
$$\set{E^{f, n}M_{\s\t}T_dX^{\kappa_d} +\W_{r,n}^{\rhd(f, \lambda)}|(\t, d, \kappa_d)\in\delta(f,\lambda)}, $$
where $\delta(f,\lambda)
  =\set{(\t, d, \kappa_d)|\t\in\Std(\lambda),d\in
                        \Dcal_{f,n}\text{ and }\kappa_d\in\Nrf
                        }$.
Then $\Delta_\s(f, \lambda)$ is a right $\W_{r, n}$-module.
\end{Prop}

\begin{proof} By Lemma~\ref{Yu2.7},   $E^{f,
n}M_{\s\t}T_dX^{\kappa_d} h +\W_{r,n}^{\rhd(f, \lambda)}$ can be
written  as an $R$-linear combination of elements
$M_{\s\t}\W_{r,n-2f}'E^{f, n}T_{e}X^{\kappa_e}+\W_{r,n}^{\rhd(f,
\lambda)}$ for  $e\in \Dcal_{f,n}$ and $\kappa_e\in\Nrf$. By
\cite[5.8d]{RX},
$$M_{\s\t}\W_{r,n-2f}'E^{f, n}\equiv E^{f, n} \epsilon_{n-2f}
(\m_{\s\t} \H_{r, n-2f}) \pmod {\W_{r, n}^{\rhd (f, \lambda)}},$$
where $\m_{\s\t}$ is given in (\ref{jme}). Finally, using
Dipper-James-Mathas's result on Murphy basis for Ariki-Koike
algebras in \cite{DJM:cyc} yields
$$M_{\s\t}\W_{r,n-2f}'E^{f, n}T_{e}X^{\kappa_e}+\W_{r,n}^{\rhd(f,
\lambda)}\in \Delta_\s(f, \lambda) .$$ So,  $\Delta_\s(f, \lambda)$
is a right $\W_{r, n}$-module.
\end{proof}
We recall the definition of cellular algebras in  \cite{GL}.
\begin{Defn}\cite{GL}\label{GL}
    Let $R$ be a commutative ring and $A$ an $R$--algebra.
    Fix a partially ordered set $\Lambda=(\Lambda,\gedom)$ and for each
    $\lambda\in\Lambda$ let $T(\lambda)$ be a finite set. Finally,
    fix $C^\lambda_{\bfs\bft}\in A$ for all
    $\lambda\in\Lambda$ and $\bfs,\bft\in T(\lambda)$.

    Then the triple $(\Lambda,T,C)$ is a \textsf{cell datum} for $A$ if:
    \begin{enumerate}
    \item $\set{C^\lambda_{\bfs\bft}|\lambda\in\Lambda\text{ and }\bfs,\bft\in
        T(\lambda)}$ is an $R$--basis for $A$;
    \item the $R$--linear map $*\map AA$ determined by
        $(C^\lambda_{\bfs\bft})^*=C^\lambda_{\bft\bfs}$, for all
        $\lambda\in\Lambda$ and all $\bfs,\bft\in T(\lambda)$ is an
        anti--isomorphism of $A$;
    \item for all $\lambda\in\Lambda$, $\bfs\in T(\lambda)$ and $a\in A$
        there exist scalars $r_{\bft\bfu}(a)\in R$ such that
        $$C^\lambda_{\bfs\bft} a
            =\sum_{\bfu\in T(\lambda)}r_{\bft\bfu}(a)C^\lambda_{\bfs\bfu}
                     \pmod{A^{\gdom\lambda}},$$
            where
    $A^{\gdom\lambda}=R\text{--span}%
      \set{C^\mu_{\bfu\bfv}|\mu\gdom\lambda\text{ and }\bfu,\bfv\in T(\mu)}$.
    \end{enumerate}
    \noindent Furthermore, each scalar $r_{\bft\bfu}(a)$ is independent of $\bfs$.
     An algebra $A$ is a \textsf{cellular algebra} if it has
    a cell datum and in this case we call
    $\set{C^\lambda_{\bfs\bft}|\bfs,\bft\in T(\lambda), \lambda\in\Lambda}$
    a \textsf{cellular basis} of $A$.
\end{Defn}

\begin{Theorem}\label{W cellular}
Let $\W_{r, n}$ be the cyclotomic Birman--Wenzl algebras over $R$.
 Then
$$\mathscr C=\bigcup_{(f, \lambda)\in \Lambda_{r, n}^+}\set{C^{(f,\lambda)}_{(\s,e, \kappa_e)(\t, d, \kappa_d)}|
            (\s, e, \kappa_e),(\t, d, \kappa_d)\in\delta(f,\lambda)
              }$$
is a cellular basis of $\W_{r, n}$ where $C^{(f,\lambda)}_{(\s, e,
\kappa_e)(\t,d, \kappa_d)}
              = X^{\kappa_e} T_e^* E^{f, n}M_{\s\t} T_d
              X^{\kappa_d}$.
 The $R$-linear map $\ast$,
 which fixes $T_i, X_j, 1\le i\le n-1$ and $1\le j\le n$ is the required anti-involution.
 In particular,
the rank of $\W_{r, n}$ is $r^n (2n-1)!!$.
\end{Theorem}
\begin{proof} This result can be proved  by arguments in the proof of
\cite[5.41]{RX}. We leave the details to the reader.  The only
difference is that we have to use Proposition~\ref{key} instead of
\cite[5.10]{RX}. Finally, we remark that  we use seminormal
representations for $\W_{r, n}$ in the proof of \cite[5.41]{RX}.
Such representations have been constructed in \cite[4.19]{RX}  for
all positive integers $r$.
\end{proof}

\begin{Remark} Yu~\cite{Yu1}  has proved that $\W_{r, n}$ is cellular under the assumption that
 $\omega_0$ is invertible. Finally, we remark that Theorem~\ref{W cellular} for all odd positive
 integers $r$ has been proved in \cite[5.41]{RX}.
\end{Remark}

 Let $F$ be an arbitrary field, which
contains the non-zero  parameters $q, u_1, \dots, u_r$ and
$q-q^{-1}$. Assume that $\Omega\cup\{\varrho\}\subset F$ is
$\bu$-admissible in the sense of the Assumption~\ref{admiss}.
\textsf{We always keep this assumption when we consider $\W_{r, n}$
over $F$ later on}. Let $\W_{r, n, F}$ be the cyclotomic
Birman--Wenzl algebra over $F$. By standard arguments, we have
$$\W_{r, n, F}\cong \W_{r, n}\otimes_R F.$$
In the remainder of this  paper, we use $\W_{r, n}$ instead of
$\W_{r, n, F}$ if there is no confusion.

By using Dipper-Mathas's Morita equivalent theorem for Ariki-Koike
algebras \cite{DM:Morita}, we can assume $u_i=q^{k_i}, k_i\in
\mathbb Z$ in the following theorem without loss of generality. See
the remark in \cite[p130]{AMR}.

\begin{Theorem}\label{simplem}
 Let $\W_{r,n}$
be the cyclotomic Birman--Wenzl algebra over $F$.
\begin{enumerate}
\item If $n$ is odd, then the  non-isomorphic
irreducible $\W_{r,n}$-modules are indexed by $(f, \lambda)$ where
$0\le f\le \lfloor \frac n2\rfloor$ and  $\lambda$ are\textsf{
$\bu$-Kleshchev multipartitions} of $n-2f$ in the sense of
\cite{AM:simples}.
\item Suppose that $n$ is an even number.
\begin{enumerate}
\item If $\omega_i\neq 0$ for some non-negative integers $i\le r-1$, then the non-isomorphic
irreducible $\W_{r,n}$-modules are indexed by   $(f, \lambda)$ where
$0\le f\le \frac n2$ and  $\lambda$ are \textsf{ $\bu$-Kleshchev
multipartitions} of $n-2f$.
\item If $\omega_i= 0$ for all non-negative integers  $i\le r-1$, then the set of all pair-wise non-isomorphic
irreducible $\W_{r,n}$-modules are  indexed by $(f, \lambda)$ where
$0\le f< \frac n2$ and  $\lambda$ are \textsf{ $\bu$-Kleshchev
multipartitions} of $n-2f$.
\end{enumerate}
\end{enumerate}
\end{Theorem}

\begin{proof} When $r$ is odd, this is  \cite[6.3]{RX}. In
general, the result still follows from the arguments in  \cite[\S
6]{RX}. The reason why Rui and Xu had to assume that $2\nmid r$ in
\cite[\S6]{RX} is that they did not have Proposition~\ref{key} for
$2\mid r$ in \cite{RX}. We leave the details to the
reader.\end{proof}

We close this section by giving a criterion on $\W_{r, n}$ being
quasi-hereditary in the sense of  \cite{CPS}.

\begin{Cor} Suppose that $\W_{r, n}$ is defined over the field $F$.
\begin{enumerate}\item Suppose that $\omega_i\neq 0$ for some
$i, 0\le i\le r-1$. Then $\W_{r, n}$ is quasi-hereditary if and only
if  $o(q^2)>n$ and $|d|\geq n$ whenever $u_iu_j^{-1}-q^{2d}=0$ and
$d\in \mathbb Z$ with $1\le i\neq j\le r$.
\item Suppose that $\omega_i=0$ for all
$i, 0\le i\le r-1$. Then $\W_{r, n}$ is quasi-hereditary if and only
if $n$ is odd and $o(q^2)>n$ and $|d|\geq n$ whenever
$u_iu_j^{-1}-q^{2d}=0$ and $d\in \mathbb Z$ with $1\le i\neq j\le
r$.\end{enumerate}
\end{Cor}
\begin{proof} Note that $\W_{r, n}$ is cellular.
By \cite[3.10]{GL}, $\W_{r, n}$  is quasi-hereditary if and only if
the non-isomorphic irreducible $\W_{r, n}$-modules  are indexed by
$\Lambda^+_{r, n}$. So, the result follows from
Theorem~\ref{simplem}. In this case, the Ariki-Koike algebras
$\H_{r, n-2f}$, $0\le f\le \lfloor n/2\rfloor$ are semisimple.
\end{proof}

\section{ The JM-basis  of  $\Delta(f, \lambda)$}
 Throughout this section, we assume that
$\W_{r, n}$ is defined  over a commutative $R$. The main purpose of
this section is to construct the JM-basis for $\W_{r, n}$.

\begin{Lemma} \label{ewe} Suppose that $n\ge 2$. We have
 $E_{n-1} \W_{r, n} E_{n-1}=E_{n-1}\W_{r,
n-2}$.\end{Lemma}
\begin{proof} Since
$\W_{r,n-2}E_{n-1}=E_{n-1}\W_{r,n-2}E_{n-2}E_{n-1}\subset
E_{n-1}\W_{r, n}E_{n-1}$, we need only to show
 the inverse inclusion.

By Lemma~\ref{Yu2.7} for $f=1$,   we need only prove that
$E_{n-1}hE_{n-1}\in \B_{r,n-2}E_{n-1}$ for $h=T_d X^{\kappa_d}$ and
$d\in \Dcal_{1, n}$. By Definition~\ref{waff}(b)(c), we can assume
$X^{\kappa_d}=X_{n-1}^k$ for some $k\in \mathbb Z$ without loss of
generality.

Note that the Birman-Wenzl algebra $\W_{1, n}$ is a subalgebra of
$\W_{r, n}$. The result for  $k=0$ follows from the corresponding
result for $\W_{1, n}$ in \cite{BirmanWenzl}.
 Assume that $k\neq 0$. We have   $i_1=n-1$ and $j_1=n$ if
$d=s_{n-1, i_1} s_{n, j_1}$. So, $d=1$. By \cite[4.21]{RX},
$E_{n-1}X_{n-1}^k E_{n-1}=\omega_{n-1}^{(k)} E_{n-1}$ for some
$\omega_{n-1}^{(k)}\in \W_{r, n-2}$.  So, $E_{n-1} \W_{r, n}
E_{n-1}\subseteq E_{n-1}\W_{r, n-2}$. \end{proof}

Using Lemma~\ref{ewe} repeatedly yields the following result.
\begin{Cor} \label{ewef} $E^{f, n}\B_{r,n}E^{f, n}=\B_{r,n-2f}E^{f, n}$, for all
 positive integers  $ f\leq \lfloor \frac{n}{2}\rfloor$.\end{Cor}

By Theorem~\ref{W cellular}, $\W_{r, n}$ is cellular over the poset
$\Lambda^+_{r, n}$ in the sense of \cite{GL}. For each $(f,
\lambda)\in \Lambda^+_{r, n}$, we have the \textsf {cell module }
$\Delta(f, \lambda)$ with respect to the cellular basis of $\W_{r,
n}$ given in Theorem~\ref{W cellular}. By definition, it is a right
$\W_{r, n}$-module which is isomorphic to $\Delta_\s(f, \lambda)$
defined  in Proposition~\ref{key}. Later on, we will identify
$\Delta(f, \lambda)$ with  $\Delta_\s(f, \lambda)$ for
$\s=\t^\lambda$. We are going to construct a $\W_{r, n-1}
$--filtration of $\Delta(f, \lambda)$ by using arguments  in
\cite{RS}.

Let  $\sigma_f\map{\H_{r, n-2f}}\W^f_{r,n}/\W_{r,n}^{f+1}$ be the
$R$-linear map defined by
\begin{equation}\label{sigma} \sigma_f(h)=E^{f,
n}\proj_{n-2f}(h)+\W_{r,n}^{f+1}\end{equation}
 for all $h\in \H_{r,
n-2f}, 1\le f\le\floor{n2}$. Here $\epsilon_{n-2f}: \H_{r,
n-2f}\rightarrow \W_{r, n-2f}/\Ef_{n-2f}$ is the algebraic
isomorphism mentioned in section~2.

Given  $\lambda\in \Lambda_r^+(n)$ and   $\mu\in \Lambda_r(n)$. A
$\lambda$-tableau $\bfS$ is of type $\mu$ if it is obtained from
$Y(\lambda)$ by inserting the entries $(k, i)$ with $i\ge 1$ and
$1\le k\le r$ such that the number of the entries in $\bfS$ which
are equal to $(k, i)$ is $\mu_i^{(k)}$.

For any  $\bfs\in \Std(\lambda)$, let $\mu(\bfs)$  be
 obtained from $\s$ by replacing each entry $m$ in
$\s$ by $(k, i)$ if $m$ is in row $i$ of the $k$-th component of
$\bft^\mu$. Then $\mu(\bfs)$ is a  $\lambda$-tableau of type $\mu$.

Given $(k, i)$ and $(\ell, j)$ in $\set{1, 2, \dots, r}\times
\mathbb N$, we say that $(k, i)<(\ell, j)$ if either $k<\ell $ or
$k=\ell$ and $i<j$. In other words, $<$ is the lexicographic order
on
 $\set{1, 2, \dots, r}\times
\mathbb N$.

Following \cite{DJM:cyc}, we say that $\bfS=(\bfS^{(1)}, \bfS^{(2)},
\dots, \bfS^{(r)})$, a $\lambda$-tableau of type $\mu$,  is
semi-standard  if
\begin{enumerate}
\item the entries in each row of each component  $\bfS^{(k)}$ of $\bfS$ increase  weakly,
\item the entries in each column of each component $\bfS^{(k)}$ of $\bfS$ increase  strictly,
\item for each  positive integer $ k\le r$ no entry in $\bfS^{(k)}$ is of
form $(\ell, i)$ with $\ell<k$.
\end{enumerate}

Let $\SStd(\lambda, \mu)$ be the set of all semi-standard
$\lambda$-tableaux of type $\mu$.  Given  $\bfS\in \SStd(\lambda,
\mu)$ and  $\t\in \Std(\lambda)$.  Motivated by \cite{DJM:cyc},
write
\begin{equation}\label{ssbasis}
M_{\bfS \bft}=\sum_{{\substack{\bfs\in \Std(\lambda)\\
\mu(\bfs)=\bfS}}} M_{\bfs\bft}.\end{equation}

\begin{Lemma} \label{rightkey} (cf. \cite[4.8, 4.11-4.13]{RS})
\begin{enumerate} \item
 For any $h\in \B_{r, n}$, we have
$$
E^{f, n} h \equiv \sum_{\substack{h_1\in \H_{r, n-2f}\\  d\in
\Dcal_{f, n}\\ \kappa_d\in \Nrf}} \sigma_f (h_1) T_d
X^{\kappa_d}\pmod {\B_{r, n}^{f+1}}.
$$
\item For each $\mu\in \Lambda_r^+(n-2f)$, let  $L^{\mu}$ be the
right $\B_{r,n} $-submodule of $\B_{r,n}^f / \B_{r,n}^{f+1}$
generated by $E^{f, n} M_{\mu} \pmod{\B_{r,n}^{f+1}}$.
 Then  $L^{\mu}$  is the free $R$-module generated by
$ \Upsilon=\{E^{f, n} M_{\bfS \t} T_d X^{\kappa_d}
\pmod{\B_{r,n}^{f+1}} \mid \bfS \in \SStd(\lambda, \mu), (\t, d,
\kappa_d )\in \delta(f, \lambda), \lambda\in \Lambda_r^+(n-2f)\}$.
\item Suppose that $(f, \lambda)\in \Lambda_{r, n}^+$ with $f>0$. If   $\bfs \in \Std(\mu)$ such that $\mu\in
\Lambda_r^+(n-2f+1)$ and $\tau = \bfs_{n-2f} \vartriangleright
\lambda$, then $E^{f, n} T_{n-1, n-2f+1} T_{d(\bfs)}^\ast M_{\mu}
\in \B_{r, n}^{\rhd(f, \lambda)}$. \item Suppose that $(f,
\lambda)\in \Lambda^+_{r, n}$ with $f>0$ and $h\in E^{f-1, n-1}
M_\lambda \B_{r,n-1} \cap \B_{r,n-1}^f$. Then
 {\small $$
 E_{n-1}T_{n-1,n-2f+1} h   \equiv  \sum_{\substack{ h_1\in \H_{r, n-2f}\\ e\in \Dcal_{f, n-1}\\
  \kappa_e \in \N_r^{f,n-1}}}  E^{f, n} M_{\lambda}
\epsilon_{n-2f}(h_1) T_{n-2f,n} T_e X^{\kappa_e}
\pmod{\B_{r,n}^{f+1}}.
 $$}

 \end{enumerate}\end{Lemma}

\begin{proof} One can use arguments in the proof of \cite[4.8]{RS} together
with Corollary~\ref{ewef}
 to verify (a). (b)-(d) can be proved by arguments in the proof of
 \cite[4.11-4.13]{RS}.
 \end{proof}

Given two multipartitions $\lambda$ and $\mu$.  We say that $\mu$ is
obtained from $\lambda$ by \textsf{adding} a box (or node) and write
$\lambda\rightarrow \mu$ if there exists a pair $(s, i)$ such that
$\mu^{(s)}_i=\lambda^{(s)}_i+1$ and $\mu^{(t)}_j=\lambda^{(t)}_j$
for $(t, j)\ne(s, i)$. In this case, we will also say that $\lambda$
is obtained from $\mu$ by \textsf{removing} a box (or node).

\begin{Defn}\label{addrem} Suppose $\lambda\in \Lambda_r^+(n-2f)$ with
$s$ removable nodes $p_1, p_2, \cdots, p_s$ and $m-s$ addable nodes
$p_{s+1}, p_{s+2}, \dots, p_m$.
\begin{itemize}\item
Let $\mu_\lambda(i)\in \Lambda_r^+(n-2f-1)$  be obtained from
$\lambda$ by removing the  box $p_i$  for $1\le i\le s$.
\item Let  $\mu_\lambda(j)\in
\Lambda_r^+(n-2f+1)$  be obtained from $\lambda$ by adding the box
$p_{j}$ for  $s+1\le j\le m$.
\end{itemize}
\end{Defn}
 We identify $\mu_\lambda(i)$ with $(f, \mu_\lambda(i))\in
\Lambda^+_{r, n-1}$ (resp. $(f-1, \mu_\lambda(i))\in \Lambda^+_{r,
n-1}$) for $1\le i\le s$ (resp. $s+1\le i\le m$). So,
$\mu_\lambda(i)\rhd\mu_\lambda(j)$ for all  $i, j$ with $1\le i\le
s$ and $s+1\le j\le m$.
 We arrange the nodes
$p_i, 1\le i\le m$  such that
\begin{equation}\label{cellfil}\mu_\lambda(i)\vartriangleright \mu_\lambda(i+1)
\quad \text{for all $i$, $1\le i\le m-1$}\end{equation} with respect
to the partial order $\trianglelefteq$ on $\Lambda^+_{r, n-1}$.

For each $\lambda=(\lambda^{(1)},  \lambda^{(2)}, \dots,
\lambda^{(r)})\in \Lambda_r^+(n)$, let $[\lambda]=[a_1, a_2, \dots,
a_r]$ such that $a_i=\sum_{j=1}^i |\lambda^{(j)}|$, $1\le i\le r$.
  Write $[\mu_\lambda(i)]=[b_1, b_2, \dots, b_r]$ for $s+1\le i\le m$. In
the later case, $\mu_\lambda(i)$ is obtained from $\lambda$ by
adding a box, say $p_i=(t, k, \lambda_k^{(t)}+1)$. We remark that
$(t, k, \ell)\in Y(\lambda)$  is in the $k$-th row, $ \ell$-th
column of the $t$-th component of $Y(\lambda)$. When $1\le i\le s$,
$\mu_\lambda(i)$ is obtained from $\lambda$ by removing the box, say
$p_i=(t, k, \lambda_k^{(t)})$. We define
\begin{equation}\label{api}
\begin{cases} a_{p_i}= a_{t-1} +\sum_{j=1}^k
\lambda_j^{(t)}, &\text{if $1\le i\le s$,}\\
b_{p_i}=b_{t-1}
+\sum_{j=1}^{k} \mu_\lambda(i)_j^{(t)},  &\text{if $s+1\le i\le m$,}\\
  \end{cases}
\end{equation}
and
\begin{equation}\label{up}
y_{\mu_\lambda(i)}^\lambda =\begin{cases}  E^{f, n} M_\lambda
T_{a_{p_i},
n},& \text{if $1\le i\le s$,}\\
E_{n-1}T_{n-1, b_{p_i}} E^{f-1, n-1}M_{\mu_\lambda(i)},
& \text{if $s+1\le i\le m$.}\\
\end{cases}
\end{equation}

For each positive integer $i\le m$,  define
\begin{equation}\label{rest} \delta(\lambda,i)
  =\{(\s, d, \kappa_d)\mid \s\in \Std(\mu_\lambda(i)),
  d\in \Dcal_{\ell,n-1},\text{ and } \kappa_d \in \Nrl \} ,\end{equation}
                        where $\ell=f$ (resp.  $\ell=f-1$)  if $1\le i\le s
$ (resp.   $s+1\le i\le m$).

In the remainder of this section, we will keep our previous notation
$\mu_\lambda(i)$. In other words, $\mu_\lambda(i)$ is obtained from
$\lambda$ by removing (resp. adding) the node $p_i$ for $1\le i\le
s$ (resp. $s+1\le i\le m$).

\begin{Theorem} \label{cellf} For any $(f, \lambda)\in \Lambda^+_{r,
n}$ with $f\ge 0$,  let $\Wmuieq$  be the $R$-submodule of
$\Delta(f, \lambda)$ generated by
 $\{ y_{\mu_\lambda(j)}^\lambda T_{d(\bft)} T_d X^{\kappa_d}
  \pmod {\W_{r, n}^{\vartriangleright {(f, \lambda)}}} \mid
(\bft, d, \kappa_d) \in \delta(\lambda,j), 1\le j\le i\}$. Then
$$(0)\subseteq S^{\gedom\mu_\lambda(1)}\subseteq \cdots \subseteq
S^{\gedom\mu_\lambda (m)}= \Delta(f,\lambda)$$ is a
$\B_{r,n-1}$--filtration of $\Delta(f,\lambda)$. Further, we have
the following $\W_{r, n-1}$-isomorphism:
$$ \Delta(\ell, \mu_\lambda(i)) \cong \Wmuieq/S^{\unrhd
\mu_\lambda(i-1)}, 1\le i\le m. $$
\end{Theorem}

\begin{proof} When $f=0$, each cell module $\Delta(0, \lambda)$ can
be considered as a cell module for $\H_{r, n}$. The result for $f=0$
has been given in \cite{AM:simples}. In the remainder of the proof,
we assume $f>0$.

 Using  arguments in the proof of \cite[4.9,
4.14]{RS}, we  can  prove that all $\Wmuieq$, $1\le i\le m$, are
$\W_{r, n-1}$-modules. Of course, we have  to use
Lemma~\ref{rightkey} instead of \cite[4.8, 4.11-4.13]{RS}. So,
$(0)\subseteq S^{\gedom\mu_\lambda(1)}\subseteq \cdots \subseteq
S^{\gedom\mu_\lambda (m)}$ is a filtration of $\B_{r,n-1}$--modules.

Let $ \phi_i:  \Delta(\ell, \mu_\lambda(i)) \rightarrow
\Wmuieq/S^{\unrhd \mu_\lambda(i-1)} $ be the $R$-linear map  sending
$E^{\ell,n-1}M_{\mu_\lambda(i)} T_{d(\t)} T_e X^{\kappa_e}\pmod
{\W_{r, n-1}^{\rhd (\ell, \mu_\lambda(i))}}$ to $
y_{\mu_\lambda(i)}^{\lambda} T_{d(\t)} T_e X^{\kappa_e} \pmod
{S^{\unrhd \mu_\lambda(i-1)}}$
 for all $ (\t, e, \kappa_e)\in \delta(\lambda, i)$.
$\phi_i$ is a $\B_{r,n-1}$--homomorphism since multiplying an
element on the left is a homomorphism of right modules.

 We claim
$\Delta(f,\lambda)=S^{\gedom\mu_\lambda(m)}$. In fact, by
Proposition~\ref{key} for $\s=\t^\lambda$ and
Definition~\ref{waff}(h), $\Delta(f, \lambda)\subset E^{f, n}
M_\lambda \B_{r, n-1} \pmod {\B_{r, n}^{\rhd (f, \lambda)}}$. Note
that $M_{\mu_\lambda(m)}=M_\lambda$. We have
$$\begin{aligned} E^{f, n} M_\lambda & =E_{n-1} T_{n-1, n-2f+1} E^{f-1, n-1}
M_{\mu_\lambda(m)} T_{n-1, n-2f+1}^\ast\\
&=y_{\mu_\lambda(m)}^\lambda T^\ast_{n-1, n-2f+1}\in
S^{\unrhd\mu_\lambda(m)}
.\\
\end{aligned} $$ Since $S^{\unrhd\mu_\lambda(m)}$ is a right $\B_{r,
n-1}$-module, $\Delta(f, \lambda)\subseteq
S^{\unrhd\mu_\lambda(m)}$.
 The inverse inclusion is trivial. This proves our claim. Counting the rank of
 $\Delta(f, \lambda)$  forces each  $\phi_i$ to be  an
 $R$-linear isomorphism.
\end{proof}

We  are going to recall the notion of $n$-updown tableaux in
\cite{AMR} in order to construct the JM-basis of $\W_{r, n}$.

Fix $(f, \lambda)\in \Lambda^+_{r, n}$. An $n$--updown
$\lambda$--tableau, or more simply an updown $\lambda$--tableau, is
a sequence $\tt=(\tt_0, \tt_1,\tt_2,\dots,\tt_n)$ of multipartitions
such that $\tt_0=\emptyset$,  $\tt_n=\lambda$ and  $\tt_i$ is
obtained from $\tt_{i-1}$ by either adding or removing a box, for
$i=1,\dots,n$. Let $\UPD_n(\lambda)$ be the set of all $n$--updown
$\lambda$--tableaux.

Given  $\tt\in \UPD_n(\lambda)$ with $(f, \lambda) \in
\Lambda_{r,n}^+$, define $f_j\in \mathbb N$ by declaring that
$\tt_j\in \Lambda_r^+(j-2f_j)$. So,  $0\le f_j\le \floor {j2}$.

Motivated by \cite{RS}, we define $\m_\tt=\m_{\tt_n}\in \B_{r, n}$
inductively by declaring that    $\m_{\tt_0}=1$ and
\begin{enumerate} \item
 $ \m_{\tt_i}=\sum_{j=a_{s,k-1}+1}^{a_{s,k}}q^{a_{s,k}-j} T_{j, a_s} \prod_{j=s}^{r-1}
(X_{a_j}-u_{j+1})T_{a_j,a_{j+1}}  T_{a_r,i} \m_{\tt_{i-1}}$ if
$\tt_i=\tt_{i-1}\cup p$ with $p=(s,k, \mu^{(s)}_k)$ and $a_{s,k} =
a_{s-1} + \sum_{j=1}^k \mu_j^{(s)}$.

\item $\m_{\tt_i}= E_{i-1} T_{i-1, b_{s,k}} \m_{\tt_{i-1}}$ if  $\tt_{i-1}=\tt_{i}\cup p$ with $p=(s,k,\nu_k^{(s)})$
and $b_{s,k}=b_{s-1} + \sum_{j=1}^k \nu^{(s)}_j$.
\end{enumerate} where $\mu=\tt_i$ and $\nu=\tt_{i-1}$ with $[\mu]=[a_1, a_2, \dots,
a_r]$ and $[\nu]=[b_1, b_2, \dots, b_r]$.

Now, we define $b_{\tt_i}$ inductively such that
$$\m_{\tt}\equiv E^{f, n}
M_{\lambda} b_{\tt_n}\pmod {\W_{r, n}^{\rhd(f, \lambda)}}.$$  We
write $\m_\lambda=E^{f, n} M_{\lambda}$. Suppose that
$\tt_{n-1}=\mu$, $[\lambda]=[a_1, a_2, \dots, a_r]$ and $[\mu]=[b_1,
b_2, \dots, b_r]$. We have $b_{\tt_0}=1$
 and {\small
\begin{equation}\label{des-b} b_{\tt_{n}}=\begin{cases} T_{a_{\ell,k},
n}b_{\tt_{n-1}}, & \text{ if
$\tt_{n}= \tt_{n-1}\cup\{ (\ell, k, \lambda^{(\ell)}_k)\}$,}\\
T_{n-1,b_{r-1}} h b_{\tt_{n-1}}, &
\text{ if $\tt_{n-1}= \tt_{n}\cup\{ (s, k, \mu^{(s)}_k)\}$,} \\
T_{n-1, b_{r, k}} \sum_{j=b_{r, k-1}+1}^{b_{r, k}} q^{b_{r, k}- j}
T_{b_{r, k}, j}b_{\tt_{n-1}},
&\text{ if  $\tt_{n-1}= \tt_{n}\cup\{ (r, k, \mu^{(r)}_k)\}$,}\\
\end{cases}
\end{equation}}
where $s\neq r$ and {\small $$ h =\prod_{j=r}^{s+2}
\{(X_{b_{j-1}}-u_j)T_{b_{j-1},b_{j-2}}\}\times
(X_{b_s}-u_{s+1})T_{b_s,b_{s,k}}\sum_{j=b_{s,k-1}+1}^{b_{s,k}}q^{b_{s,k}-j}
T_{b_{s,k}, j}.
$$}
We also use $b_{\tt}$ instead of $b_{\tt_n}$.

 For any $\ss, \tt\in \UPD_n(\lambda)$, we identify
$\ss_{i}$ (resp. $\tt_i$) with $(f_i, \ss_i)$ (resp. $(g_i, \tt_i)$)
where $(f_i, \ss_i), (g_i, \tt_i) \in \Lambda^+_{r, i}$. We write
$\ss\overset {k}\succ\tt$ if  $\ss_j\rhd\tt_j$ and $\ss_l=\tt_l$ for
$j+1\le l\le n$ and $j\ge k$. We write $\ss\succ\tt$ if there is a
positive integer  $ k\le n-1$ such that $\ss\overset {k}\succ\tt$.
In \cite{RS}, we have  verified that $\ss\succ \vv$ if $\ss\succ
\tt$ and $\tt \succ\vv$. So, $\succ $ can be refined to be a
\textsf{linear order} on $\UPD_n(\lambda)$.

There is a partial order $\unrhd$ on $\UPD_n(\lambda)$. More
explicitly, we have $\ss\unrhd\tt$ if $\ss_i\unrhd\tt_i, 1\le i\le
n$. We write $\ss\rhd\tt$ if $\ss\unrhd\tt$ and $\ss\neq \tt$.

 There is a unique element, say $\tt^\lambda\in \UPD_n(\lambda)$,
which is maximal with respect to $\succ$. More explicitly, we have
$\tt^\lambda_{2i}=\emptyset$ and $\tt^\lambda_{2i-1}=((1),
\emptyset, \dots, \emptyset)$ for $1\le i\le f$ and
$\tt^\lambda_{j}=\t^\lambda_{j-2f}$ for $2f+1\le j\le n$.

Let $E_{i, j}= E_iE_{i+1} \cdots E_{j-1}$ for  $i<j$.  If $i=j$, we
set $E_{ij}=1$.  When $i>j$, we define $E_{i,
 j}= E_{i-1}E_{i-2}\cdots E_{j}$. So,

 \begin{equation}\label{mlambda} \m_{\tt^{\lambda}} = E^{f, n}
M_{\lambda}\prod_{i=1}^f E_{n-2(f-i)-1, 2i-1}
\prod_{j=2}^r\prod_{k=1}^f (X_{2k-1}-u_j).\end{equation}

Suppose   $\tt\in \UPD_n(\lambda)$ with $(f, \lambda)\in
\Lambda^+_{r,n}$. Let
\begin{equation}\label{residue} c_\tt(k)= \begin{cases}
   u_sq^{2(j-i)},  &\text{if }\tt_k=\tt_{k-1}\cup (s,i,j),\\
   u_s^{-1}q^{-2(j-i)}, &\text{if }\tt_{k-1}=\tt_{k}\cup(s,i,j),
\end{cases}\end{equation}
 and  \begin{equation} c_\lambda(p)= \begin{cases} \label{content}
   u_sq^{2(j-i)},  &\text{if $p=(s, i, j)$ is an addable node of $\lambda$},\\
   u_s^{-1}q^{-2(j-i)},  &\text{if $p=(s, i, j)$ is a removable node of $\lambda$}.\\
\end{cases}\end{equation}

In the remainder  of this paper, unless otherwise stated, we always
use $\m_\tt$ instead of  $\m_\tt + \B_{r, n}^{\rhd (f, \lambda)}\in
\Delta(f, \lambda)$.

\begin{Prop}\label{murphybasis}\begin{enumerate}\item
 $\{\m_\tt  \mid \tt\in
\UPD_n(\lambda)\}$ is an $R$-basis of $\Delta(f, \lambda)$ for any
$(f, \lambda)\in \Lambda_{r,n}^+$. \item $\m_\tt (\prod_{i=1}^n X_i)
= \prod_{k=1}^n c_{\tt^{\lambda}}(k) \m_\tt, \forall \ \tt\in
\UPD_n(\lambda)$.
\end{enumerate}
\end{Prop}

\begin{proof} (a) follows immediately from Theorem~\ref{cellf}. In
order to prove (b), we consider $\W_{r, n}$  over the
 field of fraction of $R_0$ where
 $R_0=\mathbb Z[u_1^{\pm}, u_2^{\pm}, \cdots,
u_r^{\pm},q^{\pm},(q-q^{-1})^{-1}]$. Note that we are assuming that
 $u_1, u_2, \cdots, u_r,q$ are indeterminates.
  By the counterpart of \cite[5.3]{AMR} for $\W_{r, n}$, we have that  $\W_{r,
n}$ is split semisimple. Therefore, each cell module of $\W_{r, n}$
is irreducible. In particular, $\Delta(f,\lambda)$ is irreducible.
By Definition~\ref{waff}, we have that $\prod_{i=1}^n X_i$ is
central in $\B_{r,n}$. By Schur's Lemma, $\prod_{i=1}^n X_i$ acts on
$\Delta(f,\lambda)$ as a scalar. This enables us to consider the
special case $\tt=\tt^\lambda$ without loss of generality. By direct
computation, $$\m_{\tt^{\lambda}} X_{i}=\begin{cases}
 u_1^{(-1)^{i-1}}\m_{\tt^{\lambda}}, & \text{ if $1\le i\le 2f$,}\\
c_{\tt^{\lambda}}(i) \m_{\tt^\lambda}, & \text{ if $2f+1\le i\le n$}.\\
\end{cases}
$$
So,
 $\m_\tt (\prod_{i=1}^n X_i) = \prod_{k=1}^n c_{\tt^{\lambda}}(k) \m_\tt$. By
(a), $\m_\tt$ is an $R_0$-basis. So (b) holds over $R_0$. Finally,
we use standard arguments on base change to get  (b)  over a
commutative ring $R$.
\end{proof}

\begin{Theorem}(cf. \cite[5.12]{RS}) \label{trian}Let  $\tt \in \UPD_n(\lambda)$ with
 $(f, \lambda) \in \Lambda^+_{r,n}$. For any $k$, $1 \le k\le
n$,  there are some $\uu\in \UPD_n(\lambda)$ and $a_\uu\in R$ such
that
$$ \m_\tt X_k = c_{\tt}(k) \m_\tt+\sum_{
\uu\overset{k-1} \succ \tt} a_{\uu} \m_\uu.$$
\end{Theorem}

\begin{proof} Note that $\prod_{k=1}^n c_\tt(k)=\prod_{k=1}^n c_{\tt^\lambda}(k)$ for any
$\tt\in \UPD_n(\lambda)$.  By  Lemma~\ref{murphybasis}(b), $\m_\tt
\prod_{k=1}^n X_k =\prod_{k=1}^n c_\tt(k)\m_\tt $. We consider the
action of $\prod_{i=1}^{n-1} X_i$ on $\m_\tt$. We use the $\B_{r,
n-1}$-filtration of $\Delta(f, \lambda)$ in Theorem~\ref{cellf}. By
Lemma~\ref{murphybasis}(b),
$$
 \m_\tt\prod_{j=1}^{n-1} X_j-\prod_{j=1}^{n-1} c_\tt(j)\m_\tt\in S^{\unrhd
\mu_\lambda(i-1)}$$ where $\mu_\lambda(j), 1\le j\le m$ are defined
in Theorem~\ref{cellf} with $\mu_\lambda(i)=\tt_{n-1}$. Since
$S^{\unrhd \mu_\lambda(i-1)}$ is a right $\W_{r, n-1}$-module,
\begin{equation}\label{resn}
\m_\tt X_n c_\tt(n)^{-1}-\m_\tt=\m_\tt\prod_{j=1}^{n-1} c_\tt(j)
\prod_{j=1}^{n-1} X_j^{-1}-\m_\tt\in S^{\unrhd
\mu_\lambda(i-1)}.\end{equation}    So, Theorem~\ref{trian} holds
for $k=n$. When we deal with
 the case $k=n-1$, we consider the filtration of $\B_{r,
n-2}$-submodules of $S^{\unrhd \mu_\lambda(i)}/S^{\unrhd
\mu_\lambda(i-1)}$. Note that $S^{\unrhd \mu_\lambda(i)}/S^{\unrhd
\mu_\lambda(i-1)}\cong \Delta(\ell, \mu_\lambda(i))$ where
$\Delta(\ell, \mu_\lambda(i))$ is the cell module for $\B_{r, n-1}$
with respect to $(\ell, \mu_\lambda(i))\in \Lambda_{r, (n-1)}^+$. By
similar arguments as above we can verify the result for $k=n-2$.
Using these arguments repeatedly yields the required formula for
general $k$.
\end{proof}

 Standard arguments prove the following result (cf.
 \cite[2.7]{RS:discriminants}).

\begin{Theorem}\label{murphynaza} For each $\tt, \ss\in \UPD_{n}
(\lambda)$ with $(f, \lambda)\in \Lambda_{r, n}^+$,
 let
$\m_{\ss\tt}=b_{\ss}^\ast \m_\lambda b_{\tt}$,  where $\ast: \B_{r,
n}\rightarrow \B_{r, n}$ is the $R$-linear anti-involution which
fixes the generators $T_i, X_j$ for $1\le i\le n-1$ and $1\le j\le
n$. \begin{enumerate} \item  $\mathscr M=\{\m_{\ss \tt} \mid \ss,
\tt\in \UPD_n(\lambda), (f, \lambda)\in \Lambda^+_{r, n}\}$ is a
cellular basis of  $\B_{r, n}$ over $R$.
\item $ \m_{\ss\tt} X_k \equiv c_{\tt}(k) \m_{\ss\tt}+\sum_{
\uu\overset{k-1} \succ \tt} a_{\uu} \m_{\ss\uu} \pmod {\W_{r,
n}^{\rhd (f, \lambda)} }$.\end{enumerate}
\end{Theorem}

\begin{Remark} Note that $\succ$ is a linear order on $\UPD_n(\lambda)$. So,
 $\mathscr M$ is a JM-basis and $\{X_1, \dots, X_n\}$ is a family of JM-element in the sense of
\cite[2.4]{M:semi}.\end{Remark}

 Given two partitions $\lambda, \mu$, write $\lambda\ominus \mu$ if
either $\lambda\subset \mu$ and $\mu\setminus \lambda=p$ for some
removable node $p$ of $\mu$  or $\lambda\supset \mu$ and
$\lambda\setminus \mu=p$ for some removable node $p$ of $\lambda$.

Given an $\ss\in \UPD_n(\lambda)$ and a positive integer $k<n$. If
$\ss_k\ominus\ss_{k-1}$ and $\ss_{k+1}\ominus\ss_k$ are in different
rows and in different columns then we define, following \cite{AMR},
$\ss s_k$ to be the updown $\lambda$--tableau
$$\ss s_k=(\ss_1, \cdots,
\ss_{k-1}, \tt_{k}, \ss_{k+1}, \cdots, \ss_{n})$$ where $\tt_{k}$ is
the multipartition which is uniquely determined by the conditions
$\tt_k\ominus\ss_{k+1}=\ss_{k-1}\ominus \ss_{k}$ and
$\ss_{k-1}\ominus\tt_{k}=\ss_{k}\ominus \ss_{k+1}$. If the nodes
$\ss_k\ominus\ss_{k-1}$ and $\ss_{k+1}\ominus\ss_k$ are both in the
same row, or both in the same column, then $\ss s_k$ is not defined.

\begin{Lemma}\label{mts}(cf. \cite[5.13]{RS})
Suppose that $\tt\in\UPD_n(\lambda)$ with $\tt_{i-2}\neq \tt_i$ and
$\tt s_{i-1}\lhd\tt$.
\begin{itemize}
\item[a)]If $\tt_{i-2}\subset\tt_{i-1}\subset\tt_i$, then
$\m_{\tt}T_{i-1}=\m_{\tt s_{i-1}}+\sum_{\uu\overset{i-1}\succ \tt
s_{i-1}}a_{\uu}\m_{\uu}$ for some scalars $a_\uu\in R$.

\item[b)]If $\tt_{i-2}\supset\tt_{i-1}\subset\tt_i$ such that
$(\tilde p,\ell)>(p,k)$ where
$\tt_{i-2}\setminus\tt_{i-1}=(p,k,\nu_k^{(p)})$,
$\tt_{i}\setminus\tt_{i-1}=(\tilde p,\ell,\mu_{\ell}^{(\tilde p)})$,
$\tt_{i-2}=\nu$ and $\tt_i=\mu$, then $\m_{\tt}T_{i-1}^{-1}=\m_{\tt
s_{i-1}}+\sum_{\uu\overset{i-1}\succ \tt s_{i-1}}a_{\uu}\m_{\uu}$
for some scalars $a_\uu\in R$.
\end{itemize}
\end{Lemma}

\begin{proof} First, we assume $i=n$.  One can prove (a) by verifying
 $\m_{\tt}T_{n-1}=\m_{\tt s_{n-1}}$ via (\ref{des-b}).
 We leave the details to the reader.

 In order to prove (b), write $\tt_{n-2}\setminus\tt_{n-1}=(p,k,\nu_{k}^{(p)})$
and $\tt_{n}\setminus\tt_{n-1}=(\tilde
p,\ell,\lambda_{\ell}^{(\tilde p)})$. Let $a=a_{\tilde
p-1}+\sum_{i=1}^{\ell} \lambda_i^{(\tilde p)}$,
$c=c_{p-1}+\sum_{i=1}^k \nu_i^{(p)}$. Since either $\tilde p>p$ or
$\tilde p=p$ and $\ell>k$, we have $a\geq c$.

First, we assume $p<r$, then
$$\m_{\tt}=E^{f,n}M_{\lambda}T_{a,n}T_{n-2,c_{r-1}} A
b_{\tt_{n-2}}+\B_{r,n}^{\rhd(f,\lambda)}$$ where
\begin{equation}\label{A} A=\prod_{j=r}^{p+2}{(X_{c_{j-1}}-u_j)T_{c_{j-1},c_{j-2}}}\times
(X_{c_p}-u_{p+1})T_{c_p,c}
 \sum_{j=c_{p,k-1}+1}^{c}q^{c-j}T_{c,j}
.\end{equation} We prove (b) by induction on $\tilde p$.

If $\tilde p=r$, then $a\geq c_{r-1}$. It is routine to verify
$\m_{\tt}T_{n-1}^{-1}=\m_{\tt s_{n-1}}$.

If $\tilde p=r-1$, then $c_{r-2}\leq a\leq c_{r-1}$. We have
\begin{equation} \label{fff} \begin{aligned}
 \m_{\tt}T_{n-1}^{-1}
 & \equiv E^{f,n}M_{\lambda}T_{n-1,c_{r-1}+1}T_{a,c_{r-1}}
 \{(X_{c_{r-1}+1}-u_r)T_{c_{r-1}+1,c_{r-2}}\\ &
+\delta X_{c_{r-1}+1}E_{c_{r-1}}T_{c_{r-1},c_{r-2}} - \delta
X_{c_{r-1}+1}T_{c_{r-1},c_{r-2}}\} A\\
& \times T_{c_{r-1}+1,n-1}b_{\tt_{n-2}}
\pmod{\B_{r,n}^{\rhd(f,\lambda)}}\\
\end{aligned}\end{equation}

Since $T_{n-1,c_{r-1}+1}X_{c_{r-1}+1}T_{c_{r-1}+1,n-1}=X_{n-1},$ the
third term on the right hand of (\ref{fff})  is equal to $$h:=\delta
\sum_{j=a_{\tilde p,\ell-1}+1}^a q^{a-j}
T_{j,a}T_{a,c}E^{f,n}M_{\nu}b_{\tt_{n-2}}X_{n-1}$$ with
$\nu=\tt_{n-2}$. Since we are assuming that $\nu\rhd\lambda$, $h\in
\W_{r,n}^{\rhd(f,\lambda)}$.

The first term on the right hand side of the above equality is equal
to $\m_{\tt s_{n-1}}$. One can verify it  by  arguments in the proof
of \cite[5.13]{RS}. We leave the details to the reader.

 Finally we consider the second
term $h_1$ on the right hand side of (\ref{fff}). Since
$T_{a,c_{r-1}}X_{c_{r-1}}^{-1}=X_{a}^{-1}T_{c_{r-1},a}^{-1}$ and
$E^{f,n}T_{n-1,c_{r-1}+1}E_{c_{r-1}}T_{c_{r-1}+1,n-1}=E^{f,n}T_{c_{r-1},n}T_{n-2,c_{r-1}},$
$\delta^{-1}  h_1$ is equal to $$\begin{aligned}
&E^{f,n}M_{\lambda}X_{a}^{-1}T_{c_{r-1},a}^{-1}T_{c_{r-1},n}T_{n-2,c_{r-2}}
A b_{\tt_{n-2}}+\B_{r,n}^{\rhd(f,\lambda)}\\
 =&c_{\tt^{\lambda}}(a)^{-1}E^{f,n}M_{\lambda}\prod_{j=a}^{c_{r-1}-1}
 (T_{j}-\delta)T_{c_{r-1},n}T_{n-2,c_{r-1}}T_{c_{r-1},c_{r-2}}\\
& \times  A  b_{\tt_{n-2}}+\B_{r,n}^{\rhd(f,\lambda)}.\\
\end{aligned}$$
Note that  $\prod_{j=a}^{c_{r-1}-1} (T_{j}-\delta) \times
T_{c_{r-1},n}$ can be written as an $R$-linear combination of
$T_{\ell,n}h,$ with $a\leq\ell\leq c_{r-1}$ and $h\in
\W_{r,\ell-1}.$ So $\delta^{-1}c_{\tt^\lambda}(a)h_1$ can be written
as an $R$-linear combination of the following elements
$$
E^{f,n}M_{\lambda}T_{\ell,n}T_{n-2,c_{r-1}}T_{c_{r-1},c_{r-2}} A
b_{\tt_{n-2}}h+\B_{r,n}^{\rhd(f,\lambda)}.$$

Note that $M_{\lambda}T_w\equiv q^{l(w)}M_{\lambda} \pmod {\langle
E_1
 \rangle}$ if $w\in \mathfrak S_\lambda$. So, $M_{\lambda}T_{\ell, n-2f}\equiv q^k M_\lambda
T_{b, n-2f}\pmod {\langle E_1
 \rangle}$ for some integers $k, b$ such that $\v=\bft^\lambda
s_{b, n-2f}$ is a row standard tableau. Furthermore, since $b\ge
\ell\ge a$, $\v_{n-2f-1}\unrhd \tt_{n-1}$. If $\v$ is not standard,
we use \cite[3.15]{M:ULect} and \cite[5.8]{RX} to get
$$E^{f, n}M_{\lambda}T_{\ell,n-2f}\equiv \sum_{\bfs\in \Std(\lambda),\bfs\unrhd\v}
a_\bfs E^{f, n} M_{\lambda}T_{d(\bfs)}\pmod \Wlam
$$ for some scalars $a_\bfs\in R$. We write $d(\bfs)=s_{\ell', n-2f} d(\bfs')$
where $\bfs'$ is obtained from $\bfs$ by removing the entry $n-2f$.
Since $\bfs\unrhd \v$, $\bfs'\in \Std (\alpha)$ for $\alpha\in
\Lambda_{r}^+(n-2f-1)$ with $\alpha\unrhd \v_{n-2f-1}\unrhd
\tt_{n-1}\rhd (\tt s_{n-1})_{n-1}$. Therefore, $h_1$ can be written
as an $R$-linear combination of the elements
$$
E^{f,n}M_{\lambda}T_{\ell',n}T_{d(\s')}T_{n-2,c_{r-1}}T_{c_{r-1},c_{r-2}}
A b_{\tt_{n-2}}h \pmod \Wlam
$$ Note that
$E^{f,n}M_{\lambda}T_{\ell',n}=y_{\alpha}^{\lambda}$, and
$\alpha=\mu_{\lambda}(i)$ for some $i, 1\le i\le s$. So, the above
element can be written as  an $R$-linear combination of the elements
in $\{\m_{\ss}|\ss \in \UPD_n(\lambda), \ss_{n-1}\unrhd
\tt_{n-1}\rhd(\tt s_{n-1})_{n-1}\}$. In this case,
$\ss\overset{n-1}\succ\tt s_{n-1}$.

However, when $\tilde p<r-1,$ the first term is not equal to
$\m_{\tt s_{n-1}}.$ We will use it instead of $\m_{\tt}T_{n-1}^{-1}$
to get a similar equality for $i=c_{r-2}.$ This will enable us to
get three terms. If $\tilde p=r-2,$ we will be done since the first
term must be $\m_{\tt s_{n-1}}.$ The second and the third term can
be written as an $R$-linear combination of $\m_{\uu}$ with
$\uu\overset{n-1}\succ \tt s_{n-1}.$ In general, we have to repeat
the above procedure to get the required formula. This completes the
proof of our result under the assumption $p<r$.

Let $p=r$. Note that  $a\geq c$. It is routine to check that
$$\m_{\tt}T_{n-1}^{-1} \equiv \m_{\tt s_{n-1}} \pmod \Wlam.$$

This completes the proof of the result for $i=n$. In general, we use
Theorem~\ref{cellf} and the definition of $\succ$  to reduce the
result to the case for $i=n$.
\end{proof}

\section{Recursive formulae for Gram determinants}

In  this section, we assume that $\W_{r, n}$ is defined over a field
 $F$ such that the following assumptions hold.

\begin{Assumption} \label{separate} Assume that $\bu=(u_1, u_2, \cdots, u_r)\in F^r$ is generic
in the sense that $|d|\ge 2n$
     whenever there exists $d\in\Z$ such that either
    $u_i u_j^{\pm 1}=q^{2d}1_F$ and $i\ne j$,
    or $u_i=\pm q^{d}\cdot1_F$. We will also assume
    $o(q^2)>n$.\end{Assumption}
Suppose that  $\ss, \tt\in \UPD_n(\lambda)$.  Under the
Assumption~\ref{separate}, Rui and Xu  have proved that $\ss=\tt$ if
and only if $c_\ss(k)=c_\tt(k)$, $1\le k\le n$~\cite[4.5]{RX}. So,
Assumption~\ref{separate} is the \textsf{ separate condition } in
the sense of \cite[2.8]{M:semi}. This enables us to use standard
arguments in \cite{M:semi} to construct an orthogonal basis for
$\Delta(f, \lambda)$ as follows.

For each positive integer $k\le n$, let $$R(k)=\{ c_\tt(k)\mid
\tt\in\UPD_n(\lambda)\}.$$
 For $\ss, \tt\in \UPD_n(\lambda)$, let
 \begin{enumerate} \item
 $F_\tt
=\prod_{k=1}^n F_{\tt, k}$, \item  $f_{\ss\tt}=F_\ss \m_{\ss\tt}
F_\tt$, \item $f_\ss= \m_\ss  F_\ss \pmod {\W_{r, n}^{\rhd (f,
\lambda)}}$,
\end{enumerate}
where
\begin{equation}\label{ftk} F_{\tt, k} =\prod_{\substack{r\in \mathscr R(k)\\
c_\tt(k)\neq r}} \frac {X_k-r} {c_\tt(k)-r}.\end{equation}

 The following results hold for a general class of cellular
 algebras which have JM-bases such that the separate condition holds~\cite[\S3]{M:semi}.

\begin{Lemma}\label{fxproduct}
 Suppose that  $\tt\in \UPD_n(\lambda)$ with  $(f, \lambda)\in \Lambda^+_{r, n}$.
\begin{enumerate}\item  $f_\tt=\m_\tt + \sum_{\ss\in \UPD_n(\lambda)}
a_\ss \m_\ss$, and $\ss\succ  \tt$ if $a_\ss\neq 0$.
 \item $\m_\tt=f_\tt + \sum_{\ss\in \UPD_n(\lambda)} b_\ss
f_\ss$, and $\ss\succ \tt$ if $b_\ss\neq 0$.
\item $f_\tt X_k=c_\tt(k) f_\tt$, for any integer $k$, $1\le k\le n$.
\item $f_\tt  F_\ss =\delta_{\ss\tt} f_\tt$ for all $\ss\in
\UPD_n(\mu)$ with $(\frac {n-|\mu|}{2}, \mu)\in \Lambda^+_{r, n}$.
\item  $\set {f_\tt \mid \tt\in \UPD_n(\lambda)}$ is a basis of
$\Delta(f, \lambda)$. \item  The Gram determinants associated to
$\Delta(f, \lambda)$ defined by $\set {f_\tt \mid \tt\in
\UPD_n(\lambda)}$ and the JM-basis in Proposition~\ref{murphybasis}
are the same.
\item $\{f_{\ss\tt}\mid \ss, \tt\in \UPD_n(\lambda), (f, \lambda)\in \Lambda^+_{r, n}\}$
is an $F$-basis of  $\W_{r, n}$. Further, we have $f_{\ss\tt}
f_{\uu\vv}=\delta_{\tt\uu} \langle f_\tt, f_\tt \rangle f_{\ss\vv}$
where $\ss, \tt, \uu, \vv$ are updown tableaux and $\langle  \ ,  \
\rangle$ is the invariant bilinear form defined on the cell module
$\Delta(f, \lambda)$.
 \end{enumerate}
\end{Lemma}

By Lemma~\ref{fxproduct}(f), we can compute the Gram determinant
associated to $\Delta(f, \lambda)$ by computing each  $\langle
f_\tt, f_\tt \rangle$, for $\tt\in \UPD_n(\lambda)$.

Given two $\ss, \tt\in \UPD_n(\lambda)$ and a positive integer $k\le
n-1$. We write $\ss\simk\tt$ if $\ss_j=\tt_j$ for $1\le j\le n$ and
$j\neq k$.

\begin{Defn} \label{ett}For any $\ss, \tt\in \UPD_n(\lambda)$ and a positive integer $k\le n-1$,  define
$T_{\tt\ss}(k), E_{\tt\ss}(k)\in F$ by declaring that
$$
f_\tt T_k=\sum_{\ss\in \UPD_n(\lambda)} T_{\tt\ss}(k) f_\ss, \quad
f_\tt E_k=\sum_{\ss\in \UPD_n(\lambda)} E_{\tt\ss}(k) f_\ss.$$
\end{Defn}

 Standard arguments prove the
following result (cf. \cite[6.8--6.9]{RS}).
\begin{Lemma} \label{equi1}Suppose $\tt\in \UPD_n(\lambda)$ and $(f, \lambda)\in \Lambda^+_{r, n}$.
 \begin{enumerate} \item  $\ss\simk \tt$  if either $T_{\tt\ss}(k)\neq 0$ or
$E_{\tt\ss}(k)\neq 0$.\item $f_\tt E_k=0$ if $\tt_{k-1} \neq
\tt_{k+1}$ for any $1\le k\le n-1$.
\item Assume  $\tt_{k-1} \neq \tt_{k+1}$.
\begin{enumerate} \item If $\tt_k\ominus \tt_{k-1}$ and $\tt_k\ominus
\tt_{k+1}$ are in the same row of a component, then $f_\tt
T_k=qf_\tt$.
\item If $\tt_k\ominus \tt_{k-1}$ and $\tt_k\ominus \tt_{k+1}$ are in
the same column of a component, then $f_\tt T_k =-q^{-1}f_\tt $.
\end{enumerate}
\item Assume $\tt_{k-1}=\tt_{k+1}$.
 \begin{enumerate}
\item $f_\tt E_k=\sum_{\ss\simk\tt} E_{\tt\ss} (k) f_\ss$.  Furthermore,
$\langle f_\ss,  f_\ss\rangle E_{\tt\ss} (k)=\langle f_\tt,
f_\tt\rangle E_{\ss\tt} (k)$. \item $f_\tt T_k=\sum_{ \ss\simk\tt}
T_{\tt\ss} (k) f_\ss$. Furthermore, $T_{\tt\ss}(k)=\delta
\frac{E_{\tt\ss}(k)-\delta_{\tt\ss}}{c_\tt(k)c_\ss(k)-1}$.
\end{enumerate}
\end{enumerate}
\end{Lemma}

\begin{Lemma}\label{fxproduct1}
Suppose that $\tt\in\UPD_n(\lambda)$ with $\tt_{k-1}\neq\tt_{k+1}$
and $\tt s_k\in \UPD_n(\lambda)$. Then
$f_{\tt}T_k=T_{\tt,\tt}(k)f_{\tt}+T_{\tt,\tt s_k}(k)f_{\tt s_k}$,
with $T_{\tt,\tt}(k)=\frac{\delta
c_{\tt}(k+1)}{c_{\tt}(k+1)-c_{\tt}(k)}$. Suppose one of the
following conditions holds: \begin{itemize}
\item[(1)] $\tt_{k-1}\subset\tt_{k}\subset\tt_{k+1},$
\item[(2)] $\tt_{k-1}\supset\tt_{k}\subset\tt_{k+1}$ such that
$(\tilde p,l)>(p,i)$ where
$\tt_{k-1}\setminus\tt_{k}=(p,i,\nu_i^{(p)})$,
$\tt_{k+1}\setminus\tt_{k}=(\tilde p,\ell,\mu_{\ell}^{(\tilde p)})$,
$\tt_{k-1}=\nu$ and $\tt_{k+1}=\mu$.
\end{itemize}
Then
$$T_{\tt,\tt
s_k}(k)=\begin{cases}
1-\frac{c_{\tt}(k)}{c_{\tt}(k+1)}T_{\tt,\tt}^2(k), &\text{if
        $\tt s_k\rhd\tt$,}\\
         1, &\text{if $\tt
        s_k\lhd\tt$.}\\
        \end{cases}$$
\end{Lemma}

\begin{proof} By defining relation \ref{waff}(f),
\begin{equation}\label{ddd} f_{\tt}T_kX_k-f_{\tt}X_{k+1}T_k=\delta
f_{\tt}X_{k+1}(E_k-1).\end{equation}
 Since we are assuming that
$\tt_{k-1}\neq \tt_{k+1}$, $\ss\in\{\tt, \tt s_{k}\}$ if
$\ss\simk\tt$. Comparing the coefficients of $f_\tt$ on both sides
of (\ref{ddd}) and using Lemma~\ref{equi1}(b) yields the formula for
$T_{\tt,\tt}(k)$, as required.

First,  we  assume that $\tt \rhd\tt s_{k}$ and $\tt_{k-1}\subset
\tt_{k}\subset \tt_{k+1}$. By Lemma~\ref{fxproduct}(a), $$f_{\tt
}=\m_{\tt } +\sum_{\uu\succ \tt } a_\uu f_\uu$$ for some scalars
$a_\uu\in F$.

By Lemma~\ref{mts}(a) and Lemma~\ref{fxproduct}(b), $\m_{\tt  }
T_{k} =\m_{\tt s_k}+ \sum_{\vv\overset{k}\succ\tt s_{k}
}b_{\vv}f_{\vv}$ for some scalars $b_\vv\in R$. We claim that
$f_{\tt s_{k}}$ can not appear in the expressions of  $f_\uu T_{k}$
with non-zero coefficient. Otherwise, $\uu\overset {k}\sim \tt
s_{k}$, forcing $\uu\in \{\tt, \tt s_{k}\}$. This is a contradiction
since
 $\tt s_{k}\lhd \tt$. By Lemma~\ref{fxproduct}(b),
 the coefficient of $f_{\tt s_{k}}$ in  $f_{\tt
 } T_{k}$ is $1$.

Suppose that $\tt_{k-1}\supset\tt_{k}\subset\tt_{k+1}$. By
Lemma~\ref{mts}(b),
$$\m_{\tt }T_{k}^{-1}=\m_{\tt s_{k}
}+\sum_{\uu\overset{k}\succ\tt s_{k} }a_{\uu}\m_{\uu},$$ for some
scalars $a_\uu\in F$.  Using \ref{waff}(b) to rewrite the above
equality yields
$$
\m_{\tt }T_{k}=\m_{\tt s_{k}}+\sum_{\uu\overset{k}\succ\tt s_{k}
}a_{\uu}\m_{\uu}+\delta \m_{\tt }-\delta \m_{\tt }E_{k}.
$$

We use  Lemma~\ref{fxproduct}(b) to write the terms on the right
hand side of the above equality as a linear combination of
orthogonal basis elements. Since $\tt s_{k}\lhd \tt$, $f_{\tt
s_{k}}$ can not appear in the expression of $
\sum_{\uu\overset{k}\succ\tt s_{k} }a_{\uu}\m_{\uu}+\delta \m_{\tt
}$.

We claim that $f_{\tt s_{k}}$ can not appear in the expression of
$\m_{\tt}E_{k}$. Otherwise,  by Lemma~\ref{fxproduct}(b), we write
$\m_{\tt}=\sum_{\vv\succeq\tt}a_{\vv}f_{\vv}$. Therefore, there is a
$\vv$ such that  $f_{\tt s_{k}}$ appears in the expression of
$f_{\vv}E_{k}$  with non-zero coefficient.  So,
$\vv\overset{k}\sim\tt s_{k}$, forcing $\vv_{k-1}\neq\vv_{k+1}$.
Thus $f_{\vv}E_{k}=0$, a contradiction. This completes the proof of
our claim. Therefore, the coefficient of $f_{\tt s_{k}}$ in
${\m_{\tt}}T_{k}$ is $1$.

 Using Lemma~\ref{fxproduct}(b)
again, we write $\m_{\tt }=f_{\tt } +\sum_{\uu\succ \tt } a_\uu
f_\uu$ for some scalars $a_\uu\in F$.  If $f_{\tt s_{k}}$ appears in
the expression of $\sum_{\uu\succ \tt } a_\uu f_\uu T_{k}$, then
$f_{\tt s_{k}} $ must appear in the expression of $f_\uu T_{k} $ for
some $\uu$. So, $\tt s_{k}\overset k \sim \uu$, forcing $\uu\in
\{\tt, \tt s_{k}\}$. This contradicts the fact $\uu\succ \tt $. So,
the coefficient of $f_{\tt s_{k}}$ in $f_{\tt }T_{k}$ is $1$.

We have proved that
 \begin{equation}\label{step1} f_{\tt
}T_k=\frac{\delta c_{\tt }(k+1)}{c_{\tt }(k+1)-c_{\tt }(k)}f_{\tt
}+f_{\tt s_k},\end{equation}   if $\tt s_k\lhd \tt$ and one of
conditions (1)-(2) holds.  Multiplying $T_k$ on both sided of
(\ref{step1}) and using  \ref{waff}(b) yields
\begin{equation}\label{step2}f_{\tt s_k}T_k=f_{\tt }+\delta f_{\tt
}T_{k}-\frac{\delta c_{\tt }(k+1)}{c_{\tt }(k+1)-c_{\tt }(k)}f_{\tt
}T_{k}-\delta\rho f_{\tt }E_k.\end{equation}
 Note that
$\tt_{k-1}\neq\tt_{k+1}$. By  Lemma~\ref{equi1}(b), $f_{\tt
s_{k}}E_k=0$. Using (\ref{step1}) to simplify (\ref{step2}) and
switching the role between $\tt s_k$ and $\tt$ yields the formula
for $T_{\tt,\tt s_k}(k)$ provided  $\tt s_k\rhd \tt$ together with
one of conditions in (1)-(2) being true.
\end{proof}

Note that $\langle f_\tt T_k,  f_{\tt s_{k}}\rangle= \langle f_\tt,
f_{\tt s_{k}} T_{k}\rangle$. By Lemma~\ref{fxproduct1}, we have the
 following result immediately.

\begin{Cor}\label{recur1}
Suppose $\tt\in\UPD_n(\lambda)$ with $(f,\lambda)\in
\Lambda_{r,n}^+$ and $\tt_{k-1}\neq \tt_{k+1}$. If $\tt
s_{k}\in\UPD_n(\lambda)$, $\tt s_k\lhd\tt$ and one of the conditions
(1)-(2) in Lemma~\ref{fxproduct1} holds, then
$$\langle f_{\tt s_k}, f_{\tt s_k}\rangle=(1-\frac{\delta^2c_{\tt}(k)c_{\tt}(k+1)}{(c_{\tt}(k+1)-c_{\tt}(k))^2})
\langle f_{\tt} , f_{\tt }\rangle.$$
\end{Cor}

Let $a$ be an integer. Let $[a]_{q^2}=\frac{q^{2a}-1}{q^2-1}$. For
any partition $\lambda=(\lambda_1, \lambda_2, \dots, \lambda_k)$,
let
$[\lambda]_{q^2}!=[\lambda_1]_{q^2}![\lambda_2]_{q^2}!\cdots[\lambda_k]_{q^2}!$.
If $\lambda=(\lambda^{(1)}, \lambda^{(2)}, \dots, \lambda^{(r)})\in
\Lambda_r^+(n)$, let $[\lambda]_{q^2}!=[\lambda^{(1)}]_{q^2}!
[\lambda^{(2)}]_{q^2}! \cdots [\lambda^{(r)}]_{q^2}!$.

\begin{Lemma}(cf. \cite[6.11]{RS})\label{ttla} Suppose that $(f,\lambda)\in \Lambda_{r,n_1}^+$  and $(f, \mu)\in \Lambda^+_{r, n_2}$.
Let $[\lambda]=[a_1, a_2, \dots, a_r] $ and $[\mu]=[b_1, b_2, \dots,
b_r]$. Then
\begin{equation}\label{ftlaration}\frac{\langle f_{\tt^\lambda}, f_{\tt^\lambda}
\rangle} {\langle f_{\tt^\mu}, f_{\tt^\mu}\rangle } =\frac{
[\lambda]_{q^2} ! \prod_{j=2}^r \prod_{k=1}^{a_{j-1}}
(c_{\bft^\lambda}(k)-u_j)} { [\mu]_{q^2} ! \prod_{j=2}^r
\prod_{k=1}^{b_{j-1}} (c_{\bft^\mu}(k)-u_j)}.\end{equation}
\end{Lemma}
\begin{proof} This can be verified by arguments in the proof of \cite[6.11]{RS}. We leave the details to the reader.
\end{proof}

\begin{number1}
 Suppose that $\lambda\in \Lambda_r^+(n-2f)$. Following \cite{RS}, we define   $ \mathscr
A(\lambda)$ (resp. $\mathscr R(\lambda))$ to be the set of all
addable (resp. removable) nodes of $\lambda$. Given a removable
(resp. an addable) node  $p=(s, k, \lambda_k)$ (resp. $(s, k,
\lambda_k+1)$) of $\lambda$, define
\begin{enumerate}
\item  $\mathscr R(\lambda)^{<p} =\{(h, l, \lambda_l)\in \mathscr
R(\lambda) \mid (h,l)>(s,k) \}$, \item $\mathscr A(\lambda)^{<p}
=\{(h, l, \lambda_l+1)\in \mathscr A(\lambda) \mid (h,l)
>(s,k)\}$, \item $\mathscr {AR}(\lambda)^{\ge p} =\{(h, l,
\lambda_l)\in \mathscr R(\lambda) \mid (h,l)\le (s,k)\}\cup \{(h, l,
\lambda_l+1)\in \mathscr A(\lambda) \mid (h,l)\le
(s,k)\}$.
\end{enumerate}
\end{number1}

Following \cite{RS:discriminants}, let  $\hat \tt=(\tt_0, \tt_1,
\tt_2, \dots, \tt_{n-1})$ and $\tilde \tt=(\ss_0, \ss_1, s_2,\dots,
\ss_{n-1}, \tt_{n})$ with $\tt_{n-1}=\mu$ and $(\ss_0, \ss_1, \ss_2,
\dots, \ss_{n-1})=\tt^\mu$ for any $\tt=(\tt_0, \tt_1, \tt_2, \dots,
\tt_n)\in \UPD_n(\lambda)$. Standard arguments prove the following
result (cf. \cite[6.15]{RS}).
\begin{Prop}\label{inductionkey} Assume that  $\tt\in \UPD_n(\lambda)$ with
$(f, \lambda)\in \Lambda^+_{r,n}$. If $\tt_{n-1}=\mu$,  then $$
\langle f_\tt,\ f_\tt \rangle = \langle f_{\hat\tt},\ f_{\hat\tt}
\rangle \frac {\langle f_{\tilde \tt},\ f_{\tilde \tt}\rangle}
{\langle f_{\tt^\mu}, f_{\tt^\mu}\rangle}.$$
\end{Prop}
By Proposition~\ref{inductionkey}, we can compute $\langle f_\tt,\
f_\tt \rangle$ recursively if we know how to compute $\frac {\langle
f_{\tilde \tt},\ f_{\tilde \tt}\rangle} {\langle f_{\tt^\mu},
f_{\tt^\mu}\rangle}$. There are three cases which will be given in
Propositions~\ref{up1}, \ref{down1} and~\ref{down2}.
\begin{Prop}\label{up1}
Suppose that  $\tt\in\UPD_n(\lambda)$ with $(f,\lambda)\in
\Lambda_{r,n}^+$. If $\hat{\tt}=\t^{\mu}$ with $\tt_n=\tt_{n-1}\cup
\{p\}$ and $p=(m,k,\lambda_{k}^{(m)})$, then
\begin{equation}\label{ste1}
 \frac{\langle f_{\tt},f_{\tt}\rangle}{\langle
f_{\tt^{\mu}},f_{\tt^{\mu}}\rangle} =\frac{(-1)^{r-m}
q^{2k}}{u_m(1-q^2)}
\frac{\prod\limits_{a\in\mathscr{A}(\lambda)^{<p}}(c_{\lambda}(a)-c_{\lambda}(p)^{-1})}
{\prod\limits_{a\in\mathscr{R}(\lambda)^{<p}}(c_{\lambda}(a)^{-1}-c_{\lambda}(p)^{-1})}.
\end{equation}
\end{Prop}

\begin{proof}  Let $\lambda=[a_1,a_2,\cdots,a_r]$, and
$\tt=\tt^{\lambda}s_{a,n}$ where
$a=2f+a_{m-1}+\sum_{i=1}^k\lambda_i^{(m)}$. Note that $\tt\lhd\tt
s_{n-1}\lhd\cdots\lhd\tt s_{n,a}=\tt^{\lambda}$, and
$\tt_a\subset\tt_{a+1}\subset\cdots\subset\tt_n$. Applying
Corollary~\ref{recur1} on the pairs $\{f_{\tt^{\lambda}s_{a,j}},
f_{\tt^{\lambda}s_{a,j+1}}\}, a\leq j\leq n-1$, we have
\begin{equation}\label{ste} \langle f_{\tt},f_{\tt}\rangle=\langle
f_{\tt^{\lambda}},f_{\tt^{\lambda}}\rangle
\prod_{j=a+1}^n(1-\delta^2\frac{c_{\tt^{\lambda}}(j)c_{\tt^{\lambda}}(a)}{(c_{\tt^{\lambda}}(j)-c_{\tt^{\lambda}}(a))^2}).
\end{equation}
Simplifying (\ref{ste}) via  the definition of $c_{\tt^\lambda}(j)$
 $a\le j\le n$ together with   (\ref{ftlaration}) yields
(\ref{ste1}).\end{proof}

\begin{Prop}\label{down1}
Suppose that   $\tt\in{\UPD_n}(\lambda)$ with $\lambda\in
\Lambda_r^+(n-2f)$ and  $\tt^{\mu}=\hat{\tt}$. If
$\tt_{n-1}=\tt_n\cup\{p\}$ with $p=(s,k,\mu_k^{(s)})$ such that
$\mu^{(j)}=\emptyset$ for all integers $j, s<j\leq r$ and
$l(\mu^{(s)})=k$, then
\begin{equation}
\frac{\langle f_{\tt},f_{\tt}\rangle}{\langle
f_{\tt^{\mu}},f_{\tt^{\mu}}\rangle}=[\mu_k^{(s)}]_{q^2}E_{\tt\tt}(n-1)\prod_{j=s+1}^r(u_sq^{2(\mu_k^{(s)}-k)}-u_j)\\
\end{equation}
\end{Prop}
\begin{proof} We have
$$\begin{aligned}
&E^{f,n}T_{n-1,n-2f+1}X_{n-2f+1}^kT_{n-2f+1,n-1}F_{\tt,n}F_{\tt,n-1}E_{n-1}\\
\overset{\ref{waff}h}=&E_{n-1}\prod_{i=2}^fE_{n-2i+1,n-2i+3}X_{n-2f+1}^kT_{n-2f+1,n-1}F_{\tt,n}F_{\tt,n-1}E_{n-1}\\
\overset{\ref{waff}j}=&E_{n-1}X_{n-1}^k\prod_{i=2}^fE_{n-2i+1,n-2i+3}T_{n-2f+1,n-1}F_{\tt,n}F_{\tt,n-1}E_{n-1}\\
\overset{\ref{waff}h,i}=&E^{f,n}X_{n}^{-k}F_{\tt,n}F_{\tt,n-1}E_{n-1}.\\
\end{aligned}$$
and
$$\begin{aligned}
&E^{f,n}T_{n-1,n-2f+1}X_{n-2f+1}^kT_{n-2f,n-1}F_{\tt,n}F_{\tt,n-1}E_{n-1}\\
\overset{\ref{waff}h}=&E_{n-1}\prod_{i=2}^fE_{n-2i+1,n-2i+3}X_{n-2f+1}^kT_{n-2f,n-1}F_{\tt,n}F_{\tt,n-1}E_{n-1}\\
\overset{\ref{waff}j}=&E_{n-1}X_{n}^{-k}\prod_{i=f}^2E_{n-2i+2,n-2i}T_{n-2f,n-1}F_{\tt,n}F_{\tt,n-1}E_{n-1}\\
\overset{\ref{waff}b,c}=&E_{n-1}\prod_{i=f}^2E_{n-2i+2,n-2i}T_{n-2f,n-1}X_{n}^{-k}F_{\tt,n}F_{\tt,n-1}E_{n-1}\\
\overset {\ref{waff}h}=&E^{f,n}T_{n-1,n-2f+1}T_{n-2f,n-1}X_{n}^{-k}F_{\tt,n}F_{\tt,n-1}E_{n-1}\\
\overset{\ref{waff}c}=&E^{f,n}T_{n-2f,n-2}T_{n-1,n-2f}X_{n}^{-k}F_{\tt,n}F_{\tt,n-1}E_{n-1}\\
\overset{\ref{waff}b,c}=&E^{f-1,n-2}T_{n-2f,n-2}E_{n-1}T_{n-2}X_n^{-k}F_{\tt,n}F_{\tt,n-1}E_{n-1}T_{n-2,n-2f}\\
\end{aligned}$$
By \cite[4.27a]{RX} and Definition~\ref{waff}j,  we can write
$E_{n-1}T_{n-2}X_n^{-k}F_{\tt,n}F_{\tt,n-1}E_{n-1}$ as an $R$-linear
combination of elements $E_{n-1} g(X_1^\pm, \dots, X_{n-2}^\pm
)X_{n-2}^\ell $ where $g(X_1^\pm, \dots, X_{n-2}^\pm )$ is a
polynomial in variables $X_1^\pm,\dots, X_{n-2}^\pm $, which is in
the center of $\W_{r, n-2}$. Therefore,
$$\begin{aligned}
&E^{f-1,n-2}T_{n-2f,n-2}E_{n-1}X_{n-2}^{\ell}  g(X_1^\pm, \dots, X_{n-2}^\pm )  T_{n-2,n-2f}\\
&=E^{f,n}T_{n-2f,n-2}X_{n-2}^{\ell}T_{n-2,n-2f} g(X_1^\pm, \dots, X_{n-2}^\pm )  \\
&=X_{n-2f}^{\ell}E_{n-1}\prod_{i=f}^2E_{n-2i+2,n-2i}T_{n-2,n-2f} g(X_1^\pm, \dots, X_{n-2}^\pm )\\
&=E^{f,n}X_{n-2f}^{\ell} g(X_1^\pm, \dots, X_{n-2}^\pm ).\\
\end{aligned}$$

Note that $f_\tt =\m_\tt F_\tt$. Here we use  $\m_\tt$ instead of
$\m_\tt\pmod \Wlam$. By (\ref{des-b}), {\small
$$\begin{aligned}
f_{\tt}E_{n-1}=& E_{n-1}T_{n-1,n-2f+1}\m_{\mu}T_{n-2f+1,n-1}b_{\tt_{n-2}}F_{\tt}E_{n-1}\\
=&M_{\lambda}E^{f,n}T_{n-1,n-2f+1}\prod_{j=s+1}^r(X_{n-2f+1}-u_j)\\
& \times \sum_{i=a_{s,k-1}+1}^{n-2f+1}
 q^{n-2f+1-i} T_{n-2f+1,i}
T_{n-2f+1,n-1}F_{\tt,n}F_{\tt,n-1}E_{n-1} b_{\tt_{n-2}}\prod_{k=1}^{n-2}F_{\tt,k}\\
=& M_{\lambda}E^{f,n}T_{n-1,n-2f+1}\prod_{j=s+1}^r(X_{n-2f+1}-u_j)\\
&  (1+ T_{n-2f} \sum_{i=a_{s,k-1}+1}^{n-2f}
 q^{n-2f+1-i} T_{n-2f,i})
T_{n-2f+1,n-1}F_{\tt,n}F_{\tt,n-1}E_{n-1} b_{\tt_{n-2}}\prod_{k=1}^{n-2}F_{\tt,k}.\\
\end{aligned}$$}
By \cite[4.21]{RX} and our two equalities in the beginning of the
proof, we can find $\Phi_{\tt},\Psi_{\ell}\in F[X_1^\pm ,X_2^\pm
,\cdots,X_{n-2}^\pm ]\cap Z(\B_{r,n-2})$, $\ell\in \mathbb Z$ such
that
$$f_\tt E_{n-1}
=E^{f,n}M_{\lambda}(\Phi_{\tt}+\sum_{\ell}X_{n-2f}^{\ell}\Psi_{\ell}\sum_{i=a_{s,k-1}+1}^{n-2f}
q^{n-2f+1-i}T_{n-2f,i}) b_{\tt_{n-2}}\prod_{k=1}^{n-2}F_{\tt,k}.$$
More explicitly,  $\Phi_\tt$ is defined by (\ref{phi}) as follows:
\begin{equation}\label{phi}
E_{n-1}\prod_{j=s+1}^r(X_n^{-1}-u_j)F_{\tt,n-1}F_{\tt,n}E_{n-1}=\Phi_{\tt}E_{n-1}.
\end{equation}

Now, we use \cite[5.8]{RX} and \cite[3.7]{JM:gram} to get
$$f_{\tt}E_{n-1}=E^{f,n}M_{\lambda}b_{\tt_{n-2}}(\Phi_{\tt}+q[\lambda_k^{(s)}]_{q^2}\sum_{\ell}\Psi_{\ell}
c_{\t^{\lambda}}(n-2f)^{\ell})\prod_{k=1}^{n-2}F_{\t,k}.$$

Let $\uu\in\UPD_n(\lambda)$ such that $\uu$ is minimal in the sense
of  $\uu\overset{n-1}\sim\tt$. Then
$\m_{\uu}=E^{f,n}M_{\lambda}b_{\tt_{n-2}}\pmod \Wlam$. Therefore, $
\m_\uu X_k=c_{\tt^\lambda}(k)\m_\uu$ for all $1\le k\le n-2$. We
have
$$ f_\tt
E_{n-1}=(\Phi_{\tt,\lambda}+q[\lambda_k^{(s)}]_{q^2}\Psi_{\tt,\lambda}
)\m_{\uu}$$ where $\Psi_{\tt,\lambda} =\sum_{\ell}
\Psi_{\ell,\lambda} c_{\t^{\lambda}}(n-2f)^{\ell}$ and
$\Phi_{\tt,\lambda}$ and $\Psi_{\ell,\lambda}$ are obtained from
$\Psi_\tt$ and $\Psi_\ell$ by using $c_{\tt^\lambda}(k)$ instead of
$X_k$, $1\le k\le n-2$. By Lemma~\ref{fxproduct}(b) and
Definition~\ref{ett},
\begin{equation}\label{euu}E_{\tt\uu}(n-1)=\Phi_{\tt,\lambda}+q[\lambda_k^{(s)}]_{q^2}\Psi_{\tt,\lambda}.\end{equation}
We compute $\Phi_{\tt,\lambda}$ and   $\Psi_{\tt,\lambda}$ as
follows. By (\ref{phi}),
$$
\begin{aligned}
\Phi_{\tt,\lambda}f_{\tt}E_{n-1}=&f_{\tt}E_{n-1}\prod_{j=s+1}^r(X_n^{-1}-u_j)F_{\tt,n-1}F_{\tt,n}E_{n-1}\\
=&E_{\tt\tt}(n-1)\prod_{j=s+1}^r(c^{-1}_{\tt}(n)-u_j)f_{\tt}E_{n-1}.\\
\end{aligned}$$
When we get the last equation, we use the fact that $f_\ss F_{\tt,
n-1} F_{\tt, n}=0$ for all $\ss\in \UPD_n(\lambda)$ with
$\ss\simn\tt$ and $\ss\neq \tt$, which follows from
Lemma~\ref{fxproduct}(d). So,
\begin{equation}\label{philambda}\Phi_{\tt,\lambda}=E_{\tt\tt}(n-1)\prod_{j=s+1}^r(c^{-1}_{\tt}(n)-u_j).\end{equation}
Similarly, we can verify
\begin{equation}\label{psilambda} \Psi_{\tt,\lambda}=qE_{\tt\tt}(n-1)\prod_{j=s+1}^r(c^{-1}_{\tt}(n)-u_j).\end{equation}
By (\ref{philambda})--(\ref{psilambda}),
$$E_{\tt\uu}(n-1)=(1+q^2[\lambda_k^{(s)}]_{q^2})E_{\tt\tt}(n-1)\prod_{j=s+1}^r(c^{-1}_{\tt}(n)-u_j).$$
On the other hand, by similar  arguments for
$f_{\tt^{\lambda}\uu}f_{\uu\tt^{\lambda}}$ in \cite[6.22]{RS} for
cyclotomic Nazarov-Wenzl algebra, we can verify
$$
f_{\tt^{\lambda}\uu}f_{\uu\tt^{\lambda}}\equiv
E_{\uu\uu}(n-1)\langle f_{\vv},f_{\vv}\rangle
f_{\tt^{\lambda}\tt^{\lambda}}
 \pmod \Wlam, $$
where $\vv=(\uu_1, \uu_2, \cdots, \uu_{n-2})\in
\UPD_{n-2}(\lambda)$. So, $\langle
f_{\uu},f_{\uu}\rangle=E_{\uu\uu}(n-1)\langle
f_{\vv},f_{\vv}\rangle$. Note that

In \cite[4.7]{RX}, Rui and Xu introduced rational functions $W_k(y,
\ss)$ in variable $y$ for any $\ss\in \UPD_n(\lambda)$ such that
$$f_\ss E_{k}\frac{y}{y-X_k} E_k=E_k W_{k}(y, \ss).$$
Suppose that $\ss=\tt$. By comparing the coefficient of $f_\uu$ on
both sides of the above  equality, we have
 $$E_{\tt\uu}(n-1)E_{\uu\tt}(n-1)=E_{\tt\tt}(n-1)E_{\uu\uu}(n-1).$$
Note that $[\mu_k^{(s)}]_{q^2}=1+q^2[\lambda_k^{(s)}]_{q^2}$ and
$c_\tt(n)=u_s^{-1} q^{2(k-\mu_k^{(s)})}$.   Therefore,
$$\begin{aligned} \frac{\langle f_{\tt},f_{\tt}\rangle}{\langle
f_{\tt^{\mu}},f_{\tt^{\mu}}\rangle} =&\frac{E_{\tt\uu}(n-1)\langle
f_{\uu},f_{\uu}\rangle}{E_{\uu\tt}(n-1)\langle
f_{\tt^{\mu}},f_{\tt^{\mu}}\rangle}\\
=&\frac{E^2_{\tt\uu}(n-1)\langle
f_{\uu},f_{\uu}\rangle}{E_{\uu\uu}(n-1)E_{\tt\tt}(n-1)\langle
f_{\tt^{\mu}},f_{\tt^{\mu}}\rangle}=\frac{E^2_{\tt\uu}(n-1)\langle
f_{\vv},f_{\vv}\rangle}{E_{\tt\tt}(n-1)\langle
f_{\tt^{\mu}},f_{\tt^{\mu}}\rangle}\\
=&(1+q^2[\lambda_k^{(s)}]_{q^2})^2E_{\tt\tt}(n-1)\prod_{j=s+1}^r(c^{-1}_{\tt}(n)-u_j)^2
\frac{\langle f_{\vv},f_{\vv}\rangle}{\langle
f_{\tt^{\mu}},f_{\tt^{\mu}}\rangle}\\
\end{aligned}$$ By
Lemma~\ref{ttla}, $\frac{\langle f_{\vv},f_{\vv}\rangle}{\langle
f_{\tt^{\mu}},f_{\tt^{\mu}}\rangle} =
{[\mu_k^{(s)}]_{q^2}^{-1}}{\prod_{j=s+1}^r(u_sq^{2(\mu_k^{(s)}-k)}-u_j)^{-1}}$.
So, $$ \frac{\langle f_{\tt},f_{\tt}\rangle}{\langle
f_{\tt^{\mu}},f_{\tt^{\mu}}\rangle}
=[\mu_k^{(s)}]_{q^2}E_{\tt\tt}(n-1)\prod_{j=s+1}^r(u_sq^{2(\mu_k^{(s)}-k)}-u_j).$$
\end{proof}

\begin{Prop}\label{down2} Suppose that $\lambda=(\lambda^{(1)},\lambda^{(2)},\dots,\lambda^{(s)},\emptyset,
\dots,\emptyset)\in \Lambda_r^+(n-2f)$ and $l(\lambda^{(s)})=l$. Let
$\tt\in{\UPD_n}(\lambda)$ with $(f,\lambda)\in\Lambda_{r,n}^+$ such
that  $\hat{\tt}=\tt^{\mu}$, and $\tt_{n-1}=\tt_n\cup\{p\}$ with
$p=(m,k,\mu_k^{(m)})$ and $(m,k)<(s,l)$. Let
$\mu=[b_1,b_2,\dots,b_r]$. We define $\uu=\tt s_{n,a+1}$ with
$a=2(f-1)+b_{m-1}+\sum_{j=1}^k\mu_{j}^{(m)}$ and
$\vv=(\uu_1,\cdots,\uu_{a+1})$. Then
 \begin{equation}\label{key3}\frac{\langle
f_{\tt},f_{\tt}\rangle}{\langle f_{\tt^{\mu}},f_{\tt^{\mu}}\rangle}=
[\mu_k^{(m)}]_{q^2} E_{\vv\vv}(a)
(u_mq^{-2k}-u_m^{-1}q^{-2(\mu_k^{(m)}-k)})^{-1} A\end{equation}
where $A=   \prod_{j=m+1}^r \frac{(u_mq^{2(\mu_k^{(m)}-k)}-u_j)}{
(u_j-u_m^{-1}q^{-2(\mu_k^{(m)}-k)})} \frac{\prod_{a\in\mathscr
A(\mu)^{<p}}(c_{\mu}(a)-c_{\mu}(p))}{\prod_{b\in\mathscr
R(\mu)^{<p}}(c_{\mu}(b)^{-1}-c_{\mu}(p))}$.
\end{Prop}

\begin{proof} We have
$\tt\lhd\tt s_{n-1}\lhd\cdots\lhd \tt s_{n,a+1}=\uu$, and
$\vv=(\uu_1,\uu_2,\dots, \uu_{a+1})$. Using Corollary~\ref{recur1}
repeatedly yields \begin{equation}\label{ftu}\langle
f_{\tt},f_{\tt}\rangle=\langle f_{\uu},f_{\uu}\rangle
\prod_{j=a+2}^n(1-\delta^2\frac{c_{\uu}(j)c_{\uu}(a+1)}{(c_{\uu}(j)-c_{\uu}(a+1))^2}).\end{equation}
By Propositions~\ref{up1} and \ref{down1}, we have
\begin{equation}\label{down3}\frac{\langle f_{\uu},f_{\uu}\rangle}{\langle
f_{\tt^{\mu}},f_{\tt^{\mu}}\rangle}=E_{\vv\vv}(a)[\mu_k^{(m)}]_{q^2}\prod_{j=m+1}^r(u_mq^{2(\mu_k^{(m)}-k)}-u_j).\end{equation}
Simplifying (\ref{ftu}) via the definition of  $c_{\uu}(j)$, $a+1\le
j\le n$ together with (\ref{down3}) yields (\ref{key3}), as
required.
\end{proof}

Assume that  $(f, \lambda)\in \Lambda_{r,n}^+$ and $(l, \mu)\in
\Lambda_{r, n-1}^+$. Write
 $(l, \mu)\rightarrow (f, \lambda)$ if either $l=f$ and $\mu$ is obtained from $\lambda$  by
 removing a removable node or $l=f-1$ and $\mu$ is obtained from $\lambda$ by adding an addable node.
Assume that $\W_{r, n}$ is semisimple. By Theorem~\ref{cellf},
\begin{equation}\label{classical}\Delta(f, \lambda)\downarrow \cong \bigoplus_{(l,
\mu)\rightarrow (f, \lambda)} \Delta(l, \mu),\end{equation} where
$\Delta(f, \lambda)\downarrow$ is $\Delta(f, \lambda)$ considered as
$\W_{r, n-1}$-module. We remark that (\ref{classical}) has been
proved in ~\cite{HO:cycBMW} over $\mathbb C$.

Motivated by  \cite{RS:discriminants}, we define
$\gamma_{\lambda/\mu}\in F$  to be the scalar given by
\begin{equation}\label{gam}
\gamma_{\lambda/\mu}= \frac{\langle f_\tt, f_\tt\rangle}{\langle
f_{\tt^\mu}, f_{\tt^\mu}\rangle }  \end{equation}
 where $\tt\in
\UPD_n(\lambda)$ with $\hat \tt=\tt^\mu\in \UPD_{n-1}(\mu)$. By
\cite[5.1]{AMR},
\begin{equation} \label{rankf}\text{rank} \Delta(f, \lambda)=
\frac{r^f n!(2f-1)!!}{(2f)!\prod_{i=1}^r (a_{i}-a_{i-1})!}
\prod_{i=1}^r \frac{a_i!}{\prod_{(k, \ell)\in \lambda^{(i)}}
h_{k,\ell}^{\lambda^{(i)}}}, \end{equation} where $(f, \lambda)\in
\Lambda^+_{r, n}$ and  $[\lambda]=[a_1, a_2, \cdots, a_r]$ and
$h_{k,\ell}^{\lambda^{(i)}}=\lambda^{(i)}_k+{\lambda^{(i)}_{\ell}}'-k-\ell+1$
is the hook length of $(k, l)$ in $\lambda^{(i)}$.

Standard arguments prove the following result (cf. \cite[6.38]{RS}).

\begin{Theorem}\label{main} Let
$\W_{r,n} $ be  over $R$ where  $R=\mathbb Z[u_1^\pm, \dots,
u_r^\pm, q^\pm, \delta^{-1}]$ satisfying  the
assumption~\ref{admiss}. Let $\det G_{f, \lambda}$ be the Gram
determinant associated to the cell module $\Delta(f, \lambda)$ of
$\W_{r,n}$. Then
 \begin{equation}\label{maingram}\det G_{f,
\lambda}=\prod_{(l, \mu)\rightarrow (f, \lambda)} \det G_{l, \mu}
\cdot \gamma_{\lambda/\mu}^{\text{rank} \Delta(l, \mu)}\in
R.\end{equation}
 Furthermore, $\text{rank } \Delta(l, \mu)$ is given by
 ~\ref{rankf} and  each scalar $\gamma_{\lambda/\mu}$ can be
computed explicitly  by Proposition~\ref{up1},
Proposition~\ref{down1} and Proposition~\ref{down2}.
\end{Theorem}

We compute $E_{\ss\ss}(k)$ for any $\ss\in \UPD_n(\lambda)$ and
$1\le k\le n$. In section~4 of \cite{RX}, Rui and Xu have
constructed the seminormal representations $\Delta(\lambda)$ for
$\W_{r, n}$ where $\lambda\in \Lambda_r^+(n-2f)$. More explicitly,
$\Delta(\lambda)$ has a basis $v_\ss$, $\ss\in \UPD_n(\lambda)$. By
standard arguments (cf. \cite[3.16]{RS:discriminants}), one can
verify that $f_\ss$ constructed in the current section is equal to
$v_\ss$ up to a scalar. Therefore, $E_{\ss\ss}(k)$ can be computed
by \cite[4.12-4.13]{RX}. We list such formulae as follows. Let
$\epsilon\in \{-1, 1\}$.

 If $r$ is odd and $\varrho^{-1}=\epsilon \prod_{i=1}^r u_i$,
  then
\begin{equation}\label{ekodd}
    E_{\ss\ss}(k)= \frac{1} {\varrho c_\ss(k)}
 \big(\frac{c_\ss(k)-c_\ss(k)^{-1}}{\delta} +\epsilon \big)
                 \prod_{\alpha}\frac{c_\ss(k)-c(\alpha)^{-1}}{c_\ss(k)-c(\alpha)},
\end{equation} where $\alpha$ run over all addable and removable
nodes of $\ss_{k-1}$ with $\alpha\neq \ss_k\setminus \ss_{k-1}$.

If $r$ is even and $\varrho^{-1}=-\epsilon q^{\epsilon}\prod_{i=1}^r
u_i$
 then
\begin{equation}\label{ekeven}
    E_{\ss\ss}(k)=\frac{1}{\varrho\delta}
    \big(1-\frac{q^{-2\epsilon}}{c_\ss(k)^2})
             \prod_{\alpha}\frac{c_\ss(k)-c(\alpha)^{-1}}{c_\ss(k)-c(\alpha)},
  \end{equation} where $\alpha$ run over all addable and removable
nodes of $\ss_{k-1}$ with $\alpha\neq \ss_k\setminus \ss_{k-1}$.

By Propositions~\ref{inductionkey}, \ref{up1}, ~\ref{down1} and
\ref{down2} together with (\ref{ekodd})-(\ref{ekeven}), we have the
following result immediately.

\begin{Cor}Suppose that $(f, \lambda)\in \Lambda^+_{r, n}$. Let $[\lambda]=[a_1, a_2, \dots, a_r]$ and $\epsilon\in \{-1, 1\}$.
  Then
$$
\langle f_{\tt^\lambda}, f_{\tt^\lambda}\rangle =\frac{[\lambda]
!}{\varrho^f \delta^f} A  \prod_{j=2}^r\prod_{k=1}^{a_{j-1}}
(c_{\bft^\lambda}(k)-u_j) \prod_{j=2}^r (u_1-u_j)^f
(u_1-u_j^{-1})^f, $$ where
$$A=\begin{cases} (u_1^{-1}+q^{-\epsilon})^f (-u_1^{-1}
+q^{\epsilon})^f, &\text{if $2\nmid r$  and $\varrho^{-1}=\epsilon \prod_{i=1}^r u_i $}, \\
(u_1+q^\epsilon)^f (u_1-q^\epsilon)^f u_1^{-2f},
 &\text{if $2\mid r$  and $\varrho^{-1}=\epsilon q^{-\epsilon}\prod_{i=1}^ru_i$}.\\
\end{cases}$$
\end{Cor}

Given an multi-partition of $\lambda$. We denote $\mu$ by
$\lambda\cup p$ ( resp. $\lambda/p$) if $Y(\mu)$ is obtained from
$\lambda$ by adding (resp. removing ) the addable (resp. removable)
node $p$. Let $p=(i, j, k)$ be the node which is in the jth-row, kth
column of ith component of $Y(\lambda)$. We define $p^+=(i, j, k+1)$
and $p^-=(i, j+1, k)$.

 In the remainder of this section, we assume that
$$R=\Z[u_1^{\pm},u_2^{\pm},\cdots,u_r^{\pm},q^{\pm 1}, \delta^{-1}
]$$ such that the assumption~\ref{admiss} holds. Let $R_1$ be the
multiplicative sub-semigroup of $R$ generated by $1$, $u_i^\pm,
q^\pm,  \delta^\pm$  and $u_iu_j^{-1}-q^{2d}$ for integers $i, j, d$
with $|d|<n$ and $ 1\leq i ,j\leq r$. Let $F_1$ be the field of
fraction of $R_1$.

\begin{Theorem}\label{line1111} Suppose $\lambda \in \Lambda_r^+ (n-2)$.
Let $r_{\lambda, p, \tilde p}=\dim\Delta(0, \lambda\cup p\cup \tilde
p)$ if $\lambda\cup p\cup \tilde p$ is an multipartition. If
$2\nmid r$ and $\varrho^{-1}=\epsilon \prod_{i=1}^r u_i$, we define
$$B=\prod_{\lambda\cup p\cup p^+ \in
\Lambda_r^+(n)} (c_\lambda(p)-\epsilon q^{-1})^{r_{\lambda, p,
p^+}}\prod_{\lambda\cup p\cup p^- \in
\Lambda_r^+(n)}(c_\lambda(p)+\epsilon q)^{r_{\lambda, p, p^-}}.$$
Otherwise, we define
$$B=\begin{cases} \prod_{\lambda\cup p\cup p^- \in \Lambda_r^+(n)}
(c_\lambda(p)^2-q^2)^{r_{\lambda, p, p^-}},
& \text{if $2\mid r, \varrho^{-1}=q^{-1}\prod_{i=1}^r u_i$,}\\
\prod_{\lambda\cup p\cup p^+ \in \Lambda_r^+(n)}
(c_\lambda(p)^2-q^{-2})^{r_{\lambda, p, p^+}}, & \text{if $2\mid r,
\varrho^{-1}=-q \prod_{i=1}^r u_i$.}
\end{cases}$$
Then there is an $A\in R_1$ such that \begin{equation}
\label{line1}\det G_{1,\lambda}= A B \prod_{\substack{p,\tilde p\in
\mathscr A(\lambda) }}
 (c_\lambda(p)c_\lambda(\tilde p)-1)^{\dim \Delta(0, \lambda \cup p \cup \tilde p)}.\end{equation} \end{Theorem}

\begin{proof} Suppose that there are
 $s$ (resp. $m-s$)  addable (resp. removable) nodes $p_1, p_2, \cdots, p_s$
 (resp. $p_{s+1}, p_{s+2}, \cdots, p_m$) in $Y(\lambda)$.
 Let $$\mu[i]=\begin{cases} \lambda\cup p_i, & \text{ if $1\le i\le s$,}\\
 \lambda/ p_i, & \text{ if $s+1\le i\le m$}.\\
 \end{cases}$$

We need (\ref{t1})--(\ref{t2}) which can be verified directly.
Suppose $s+1\le k\le m$.
\begin{equation}\label{t1}
\begin{aligned} & \{(p, \tilde p )\mid p, \tilde p\in \mathscr A(\mu[k]), p\neq \tilde p \}\\ = & \{(p,
\tilde p)\mid p, \tilde p\in \mathscr A(\mu[k])\cap \mathscr
A(\lambda), p\neq
\tilde p\}\cup \{(p, p_k)\mid p\in \mathscr A(\mu[k] )\}\\
\end{aligned} \end{equation}
and
 \begin{equation}\label {t2} \begin{aligned} & \{(p_i,
p_k)\mid s+1\le k\le m, 1\le i\le s\}\\ = &\cup_{k={s+1}}^m \{(p,
p_k)\mid p\in \mathscr A(\mu[k] )\} \cup \cup_{k=s+1}^m \{(p_k,
p_k^+), (p_k,  p_k^{-})\}.\end{aligned}\end{equation}

Now, we prove the result by induction on $n$. It is  routine to
check (\ref{line1}) for the case  $n=2$. Suppose $n\geq 3$. By
Theorem~\ref{main},
\begin{equation}\label{1}\det G_{1, \lambda}=\prod_{i=1}^s \det
G_{0, \mu[i]}\cdot \gamma_{\lambda/\mu[i]}^{\dim \Delta(0,
\mu[i])}\prod_{j=s+1}^m \det G_{1, \mu[j]}\cdot
\gamma_{\lambda/\mu[j]}^{\dim \Delta(1, \mu[j])}\end{equation}

 By
Proposition~\ref{up1},  $\det G_{0, \mu[i]}\in R_1$ and
$\gamma_{\lambda/\mu[j]}\in F_1$  for $1\le i\le s$ and  $s+1\le
j\le m$. Suppose $1\le i\le s$. By
Propositions~\ref{down1},~\ref{down2},
\begin{equation}\label{r}\gamma_{\lambda/\mu[i]}= C D
\frac{\prod_{1\le j\neq i\le
s}(c_\lambda(p_i)c_\lambda(p_j)-1)}{\prod_{s+1\le k\le
m}(c_\lambda(p_i)-c_\lambda(p_k))} £¬\end{equation} where  $C\in
F_1$ and
 $$D=\begin{cases} (c_\lambda(p_i)+\epsilon q )(c_\lambda(p_i)-\epsilon q^{-1}),
 & \text{if~$2\nmid r, \varrho^{-1}=\epsilon \prod_{i=1}^r u_i$,}\\
c_\lambda(p_i)^2-q^{2\epsilon}, & \text{if $2\mid r,
\varrho^{-1}=\epsilon q^{-\epsilon}\prod_{i=1}^r u_i.$}
\end{cases}$$

By induction assumption, $\det G_{1,\mu[j]}$ can be computed by
(\ref{line1}) if $s+1\le j\le m$. We rewrite the terms on the right
hand side of (\ref{1}) so as to get $(c_\lambda(p)c_\lambda(\tilde
p)-1)^{r_{\lambda, p, \tilde p}}$ in $\det G_{1,\lambda}$. In fact,
this follows from (\ref{t1}) and the classical branching rule for
$\Delta(0, \lambda \cup p \cup \tilde p)$. Now, (\ref{line1})
follows from similar computation together with
(\ref{t1})-(\ref{t2}).
\end{proof}

\section{Induction and Restriction}
In this section, we consider $\W_{r, n}$ over a field $F$.

Let $\W_{r, n}$-mod be the category of right $\W_{r, n}$-modules. We
define two functors
$$\F_n: \W_{r, n}\text{-mod}\rightarrow \W_{r, n-2}\text{-mod},
\text{ and } \G_{n-2}: \W_{r, n-2}\text{mod} \rightarrow \W_{r,
n}\text{-mod}$$ such that
$$ \mathcal F_n(M) = M E_{n-1}  \text{ and } \mathcal G_{n-2}(N)=N _{\W_{r, n-2}}\otimes  E_{n-1} \W_{r,
n},$$ for all right $\W_{r, n}$-modules $M$ and right $\W_{r,
n-2}$-modules $N$. By Lemma~\ref{ewe}, $\F_n$ and $\G_{n-2}$ are
well-defined. For the simplification of notation, we will omit the
subscripts of $\mathcal F_n$ and $\mathcal G_{n-2}$ later on.

\begin{Lemma} \label{fgfunctor} Suppose that $(f, \lambda)\in \Lambda_{r, n}^+$
and $(\ell, \mu)\in \Lambda_{r, n+2}^+$.
\begin{enumerate}
\item $\F\G=1$.
 \item $\G(\Delta(f, \lambda) )=\Delta(f+1,
\lambda)$.
\item $\F(\Delta(f, \lambda ))=\Delta(f-1, \lambda)$.
\item As right $\W_{r, n}$-modules,
$\text{Hom}_{\W_{r, n+2}} ( E_{n+1}\W_{r, n+2} , \Delta(\ell, \mu))
\cong \Delta(\ell , \mu) E_{n+1}$.
\item $\text{Hom}_{\W_{r, n+2} } (\G(\Delta(f, \lambda)), \Delta(l,
\mu))\cong \text{Hom}_{\W_{r, n} } (\Delta(f, \lambda), \F(\Delta(l,
\mu))$ as $F$-modules.
\end{enumerate}
\end{Lemma}

\begin{proof} (a) follows from Lemma~\ref{ewe}, immediately.
 By standard arguments, we define
$\psi: \Delta(f, \lambda) \otimes E_{n+1} \W_{r, n+2}\rightarrow
\Delta(f+1, \lambda)$ such that $$\psi(({E^{f, n} M_\lambda}+\W_{r,
n}^{\rhd(f, \lambda)}) \otimes E_{n+1} h)=E^{f+1, n+2} M_\lambda
 h+\W_{r, n+2}^{\rhd(f+1, \lambda)}$$ for  $h\in \W_{r, n+2} $.
 Since  $E^{f+1, n+2}
 M_\lambda$ generates $\Delta(f+1, \lambda)$ as $\W_{r, n+2}$-module,  $\psi$ is an epimorphism.
Note that $E^{f, n} =E^{f, n} E^{f, n-1} E^{f, n}$. We have
$$\Delta(f, \lambda) \otimes E_{n+1} \W_{r, n+2}=(M_\lambda E^{f, n}
E^{f, n-1}+\W_{r, n}^{\rhd(f, \lambda)}) \otimes E^{f+1, n+2} \W_{r,
n+2}.$$ By Lemma~\ref{Yu2.7}, $E^{f+1, n+2} \W_{r, n+2}$ can be
written as $F$-linear combination of elements in $\W_{r, n-2f}
E^{f+1, n+2} T_d X^{\kappa_d}$ where $d\in \Dcal_{f+1, n+2}$ and
$\kappa_d\in \mathbb N_r^{f+1, n+2}$. By \cite[5.8]{RX},  {\small
$$( M_\lambda E^{f, n} E^{f, n-1}+\W_{r, n}^{\rhd(f, \lambda)})
\otimes \W_{r, n-2f} E^{f+1, n+2}=(M_\lambda E^{f, n} \H_{r, n-2f}
+\W_{r, n}^{\rhd(f, \lambda)}) \otimes E_{n+1}.$$}
 Therefore, $\dim_F
(\Delta(f, \lambda) \otimes E_{n+1} \W_{r, n+2})\leq \dim_F
\Delta(f+1, \lambda)$. So, $\psi$ is injective. This completes the
proof of (b). (c) follows from (a)-(b), immediately.

We define the $F$-linear map $\phi: \text{Hom}_{\W_{r, n+2}} (
E_{n+1}\W_{r, n+2},  \Delta(\ell , \mu))\rightarrow  \Delta(\ell,
\mu) E_{n+1}$ such that $\phi(f)=f(E_{n+1})$, for $f\in
\text{Hom}_{\W_{r, n+2}} ( E_{n+1}\W_{r, n+2},  \Delta(\ell ,
\mu))$. Note that $f(E_{n+1})\in \Delta(\ell, \mu) E_{n+1}$. So,
$\phi$ is an epimorphism. Note that any $f\in \text{Hom}_{\W_{r,
n+2}} ( E_{n+1}\W_{r, n+2},  \Delta(\ell , \mu))$ is determined
uniquely by $f(E_{n+1})$. So, $\phi$ is injective. This proves (d).
Finally,
 (e) follows
from adjoint associativity and (d).
\end{proof}

Given two $\W_{r, n}$-modules $M, N$.  Let $\langle M,
N\rangle_n=\dim_F \text{Hom}_{\W_{r, n}}(M, N)$. By
Lemma~\ref{fgfunctor}(e), we have the following result immediately.

 \begin{Theorem}\label{cons} Given $(f, \lambda)\in \Lambda_{r,
n+2}^+$ and $(\ell, \mu)\in \Lambda_{r, n+2}^+$ with $f\ge 1$. Then
 $\langle \Delta(f, \lambda),
\Delta(\ell, \mu)\rangle_{n+2} = \langle \Delta(f-1, \lambda),
\Delta(\ell-1, \mu)\rangle_n $.
 \end{Theorem}

\section{A criterion on $\W_{r, n}$ being semisimple}
 In this section, we consider $\W_{r, n}$ over a field $F$.  The main purpose of this section is to give a
necessary and sufficient condition for $\W_{r, n}$ being semisimple
over $F$.

In Propositions~\ref{empty}-\ref{ss1}, we assume
   $o(q^2)>n$ and
  $|d|\geq n$ whenever $u_iu_j^{-1}-q^{2d}=0$ and $d\in \mathbb
   Z$. So, $\H_{r, n}$ is semisimple over $F$~\cite{A:semi}.
By  Theorem~\ref{line1111}, we describe explicitly when $\det G_{1,
\lambda}\neq 0$ for all $\lambda\in\Lambda_n$ where $\Lambda_n$ is
defined in Definition~\ref{lamb}.

\begin{Prop} \label{empty} $G_{1,\emptyset}\neq 0$ if and only if the  following
conditions hold:
\begin{enumerate} \item $u_iu_j-1\neq 0$ for all $1\le i\neq j\le r$,
\item  $u_i\not\in \{-\epsilon q, \epsilon q^{-1}\} $ if  $2\nmid r$ and
$\varrho^{-1}=\epsilon\prod\limits_{i=1}^ru_i$.
\item $u_i\not\in \{-q^\epsilon, q^\epsilon \}$ if  $2|r$ and
 $\varrho^{-1}=\epsilon q^{-\epsilon}\prod\limits_{i=1}^ru_i$.
 \end{enumerate}
   \end{Prop}

\begin{Prop}\label{line23}
 Suppose that $n\ge 3$. Let
$\lambda\in\Lambda_r^+(n-2)$ with  $\lambda^{(m)}=(n-2)$  for some
positive integer $ m\le r$. $\det G_{1, \lambda}\neq 0$ if and only
if the following conditions hold:\begin{enumerate}
\item $u_m\not\in \{q^{3-n}, -q^{3-n}\}$,
\item  $u_iu_m\not\in \{q^{4-2n},
q^2\}$, for all $1\le i\le r$ and $i\neq m$
\item  $u_iu_j\neq 1$ for all $m\not\in\{i, j\}$ and $i\neq j$.
\item $u_m\not\in \{-\epsilon q^3, \epsilon q^{3-2n}, \epsilon q\} $
and $u_i\not\in \{-\epsilon q, \epsilon q^{-1}\}$ for all $i\neq m$
if $2\nmid r$  and $\varrho^{-1}=\epsilon\prod_{j=1}^ru_j$.
\item $u_m\not\in \{-q^{3}, q^3\}$
and $u_i\not\in \{q, -q\}$ for all $i\neq m$ if $2\mid r$ and
$\varrho^{-1}=q^{-1}\prod_{j=1}^ru_j$.
\item $u_m\not\in \{-q^{3-2n}, q^{3-2n}, -q, q\}$  and $u_i\not\in
\{q^{-1}, -q^{-1}\}$ if $2\mid r$ and
$\varrho^{-1}=-q\prod_{j=1}^ru_j$.
\end{enumerate}
\end{Prop}

\begin{Lemma} \label{line234} Suppose that $n\ge 3$. Let $\epsilon=\pm 1$. Let
$\lambda\in\Lambda_r^+(n-2)$ with $\lambda^{(m)}=(1^{n-2})$. $\det
G_{1, \lambda}\neq 0$ if and only if  the following conditions hold:
\begin{enumerate} \item
$u_m\not\in \{q^{n-3}, -q^{n-3}\}$,
\item  $u_iu_m\not\in \{q^{2n-4},
q^{-2}\}$, for all $1\le i\le r$ and $i\neq m$
\item  $u_iu_j\neq 1$ for all $m\not\in\{i,j\}$ and $i\neq j$.
\item $u_m\not\in \{\epsilon q^{-3}, -\epsilon q^{2n-3}, -\epsilon q^{-1}\} $
and $u_i\not\in \{-\epsilon q, \epsilon q^{-1}\}$ for all $i\neq m$
if $2\nmid r$  and $\varrho^{-1}=\epsilon\prod_{j=1}^ru_j$.
\item $u_m\not\in \{ -q^{2n-3}, q^{2n-3}  -q^{-1}, q^{-1}\}$
and $u_i\not\in \{q, -q\}$ for all $i\neq m$ if $2\mid r$ and
$\varrho^{-1}=q^{-1}\prod_{j=1}^ru_j$.
\item $u_m\not\in \{-q^{-3}, q^{-3}\}$  and $u_i\not\in
\{q^{-1}, -q^{-1}\}$ if $2\mid r$ and
$\varrho^{-1}=-q\prod_{j=1}^ru_j$.
\end{enumerate}
\end{Lemma}

\begin{Defn}\label{lamb}Fix positive  integers  $r$ and $n$. let
$$
\Lambda_n=\bigcup_{k=2}^n \{\lambda\in \Lambda^+_{r}(k-2)\mid
\lambda^{(i)}\in \{(k-2), (1^{k-2})\}\text{ for some $i, 1\le i\le
r$}\}$$
\end{Defn}

\begin{Prop}\label{ss1} Suppose that $r\ge 2$ and $n\ge 2$.
\begin{enumerate} \item  Assume $\det G_{1, \emptyset}\neq 0$.  Then $\prod_{\lambda\in \Lambda_n}
\det G_{1,\lambda}\neq 0$ if and only if   $\B_{r,n}$ is (split)
semisimple over $F$.
\item  $\B_{r,n}$ is not semisimple over $F$  if
$\det G_{1, \emptyset}= 0$.
\end{enumerate}
\end{Prop}

\begin{proof}  By Propositions~\ref{empty}--\ref{line234},
$\prod_{\lambda\in \Lambda_n\setminus\Lambda_{n-1}} \det G_{1,
\lambda}=0$ if $\det G_{1, \emptyset}=0$. This proves (b).

 We are going to prove (a) by
induction on $n$. When $n=2$, there is nothing to be proved. We
assume  $n\ge 3$ in the remainder of the proof.

In \cite{GL}, Graham and Lehrer proved that a cellular algebra is
(split) semisimple if and only if  no Gram determinant associated to
a cell module which is defined by a cellular basis is equal to zero.
We use it frequently in the proof of this proposition.

 $(\Longrightarrow)$  If   $\W_{r,n}$ is not semisimple, then  $\det
G_{f,\lambda}=0$ for some $(f,\lambda)\in \Lambda^+_{r, n}$. Under
our assumption, $\H_{r, n}$ is semisimple. Since each cell module
$\Delta(0, \lambda)$ for $\W_{r, n}$ can be considered as the cell
module of $\H_{r, n}$ with respect to $\lambda$. So,  $\det G_{0,
\lambda}\neq 0$ for all $\lambda\in \Lambda_r^+(n)$. Therefore, we
can assume that $f>1$.

Take an irreducible module $D^{\ell, \mu}\subset \Rad \Delta(f,
\lambda)$. By general theory  about cellular algebras, we know that
$\ell\le f$. When $\ell>1$, we  use Theorem~\ref{cons} to get a
non-zero $\W_{r, n-2}$-homomorphism from $\Delta(\ell-1, \mu)$ to $
\Delta(f-1, \lambda)$. So, $\W_{r, n-2}$ is not semisimple. This
contradicts to our assumption since $\Lambda_{n-2}\subset
\Lambda_n$. If $\ell=0$, then there is a non-zero homomorphism from
$\text{Ind}_{\W_{r, n-1}}\Delta(0, \mu/p)$ to $\Delta(f, \lambda)$
where $p$ is a removable node of $\mu$ and $\mu/p$ is obtained from
$\mu$ by removing the removable node $p$. Here we use classical
branching rule for $\Delta(0, \mu/p)$ since we are assuming that
$\H_{r, n}$ is semisimple. By Theorem~\ref{cellf}, there is a $(k,
\alpha)\in \Lambda_{r, n-1}$ with $(k, \alpha)\rightarrow (f,
\lambda)$ such that $\Delta(0, \mu/p)$ is a composition factor of $
\Delta(k, \alpha)$. Since we are assuming that $f>1$, $k\geq f-1>0$.
So, $(0, \mu/p)\neq (k, \alpha)$. Therefore, $\W_{r, n-1}$ is not
semisimple. This contradicts our induction assumption again.

 $(\Longleftarrow)$ By assumption, $\det G_{1, \lambda}\neq 0$ for
 all $\lambda\in \Lambda_n\setminus\Lambda_{n-1}$.
Suppose that $\det G_{1, \lambda}=0$ for $\lambda\in \Lambda_{n-1}$.
We can find an irreducible module $D^{\ell, \mu}\subset \Rad
\Delta(1, \lambda)$. We have $\ell=0$. Otherwise, since $\ell\le 1$,
we have $\ell=1$.  By Theorem~\ref{cons}, $\lambda=\mu$, a
contradiction.

If $n-2-|\lambda|=2a$ for some  $a\in \mathbb N$, we can use
Theorem~\ref{cons} to get a non-zero homomorphism from $\Delta(a,
\mu)$ to $\Delta(1+a, \lambda)$. So, $\det G_{1+a, \lambda}=0$,
forcing $\W_{r, n}$ not being semisimple, a contradiction.

Suppose  $n-2-|\lambda|$ is odd.  By Theorem~\ref{line1111}, we can
find a suitable multipartition, say $\tilde\lambda$ which is
obtained from $\lambda$ by adding an addable node, such that  $\det
G_{1, \tilde \lambda}=0$. First,   we  assume that $\lambda\in
\Lambda_r^+(k-2)$ with $\lambda^{(m)}=k-2$ and $k\le n-1$ without
loss of generality. By Proposition~\ref{line23}, either $u_i\in
\{q^a, -q^b\} $ or $u_iu_j=q^c$ for some $1\le i\neq j\le r$ and
some integers $a, b, c$. In the first case, we add a box on
$\lambda^{(j)}$ with $j\neq i$. In the remainder case, we define
$\tilde \lambda^{(m)}=(k-2, 1)$ (resp. $\tilde\lambda^{(m)}=(k-1)$)
if $u_i u_m=q^{4-2k}$ (resp. otherwise).
 In each case,  $\tilde\lambda\in \Lambda_r^+(k-1)$ and
 $\det G_{1, \tilde
\lambda}=0$. Since $n-2-|\tilde\lambda|$ is a non-negative even
number, we get a contradiction by our previous arguments.

By similar arguments, we get a contradiction if we assume
$\lambda\in \Lambda_r^+(k-2)$. We leave the details to the reader.
\end{proof}

For convenience, we define
\begin{equation} Q_{r,\varrho}=\begin{cases}
\{-\epsilon q, \epsilon q^{-1}\}, &\text{if
$2\nmid r, \varrho^{-1}=\epsilon \prod\limits_{i=1}^r u_i$, }\\
\{- q^{\epsilon},q^{\epsilon}\}, &\text{if
$2| r, \varrho^{-1}=\epsilon q^{-\epsilon} \prod\limits_{i=1}^ru_i$, }\\
\end{cases}
\end{equation}

and

 {\small \begin{equation} S_{r,\varrho}=\begin{cases}
\cup_{k=3}^n \{\pm q^{3-k}, \pm q^{k-3}, \epsilon q^{3-2k},,
-\epsilon q^{2k-3}\}, &\text{if
$2\nmid r, \varrho^{-1}=\epsilon \prod\limits_{i=1}^ru_i$, }\\
\cup_{k=3}^n \{\pm q^{3-k}, \pm q^{k-3},  \pm q^{(2k-3)\epsilon} \},
&\text{if
$2| r, \varrho^{-1}=\epsilon q^{-\epsilon}\prod\limits_{i=1}^ru_i$. }\\
\end{cases}
\end{equation}}

\begin{Theorem}\label{semisimple} Let $n\ge 2$ and $r\ge 2$.
 Let $\W_{r, n}$ be defined over the
field $F$ which contains non-zero $u_i, 1\le i\le r$, $q, q-q^{-1}$
such that the assumption~\ref{admiss} holds.
\begin{enumerate} \item
If either $u_i-u_j^{-1}=0$ for different  positive integers $i, j\le
r$ or $u_i\in Q_{r, \varrho}$ for some positive integer $i\le r$,
then $\W_{r,n}$ is not semisimple.
\item  Assume $u_i-u_j^{-1}\neq 0$ for all  different  positive integers $i, j\le
r$ and $u_i\not\in Q_{r,\varrho}$ for all  positive integers $ i \le
r$.
\begin{itemize} \item[(1)] $\W_{r,2}$ is
semisimple if and only if  $o(q^2)>2$ and $|d|\ge 2$ whenever
$u_iu_j^{-1}=q^{2d}$ for any $1\le i < j\le r$ and $d\in \mathbb Z$.

 \item[(2)] Suppose  $n \ge 3$ . Then
$\W_{r,n}$ is semisimple if and only if \begin{itemize}\item [(a)]
$o(q^2)>n$, \item [(b)] $|d|\ge n$ whenever $u_iu_j^{-1}=q^{2d}$ for
any $1\le i < j\le r$ and $d\in \mathbb Z$, \item[(c)] $u_i\not\in
S_{r,\varrho}$,
\item [(d)] $u_iu_j \not\in \cup_{k=3}^n\{q^{4-2k},q^{2k-4}\}$ for all different
positive integers $i, j\le r$.
\end{itemize}
\end{itemize}
\end{enumerate}
\end{Theorem}
\begin{proof} Each cell module $\Delta(0, \lambda)$ for
$\lambda\in\Lambda_r^+(n)$ can be considered as the cell module of
$\H_{r, n}$. So, $\W_{r, n}$ is not semisimple over $F$ if $\H_{r,
n}$ is not semisimple. Therefore, we can assume $\H_{r, n}$ is
semisimple when we discuss the semisimplicity of $\W_{r, n}$. Now,
the result follows from Ariki's result on $\H_{r, n}$ being
semisimple in \cite{A:semi} together with
Propositions~\ref{empty}-\ref{ss1}.
\end{proof}

When $r=1$, Theorem~\ref{semisimple} has been proved in
\cite[5.9]{RS1}. We remark that  the notation $r$ (resp. $\omega$)
in \cite[1.1]{RS1} is the same as $\rho^{-1}$ (resp. $\delta$) in
the current paper.

\providecommand{\bysame}{\leavevmode ---\ } \providecommand{\og}{``}
\providecommand{\fg}{''} \providecommand{\smfandname}{and}
\providecommand{\smfedsname}{\'eds.}
\providecommand{\smfedname}{\'ed.}
\providecommand{\smfmastersthesisname}{M\'emoire}
\providecommand{\smfphdthesisname}{Th\`ese}

\end{document}